\documentclass[twoside,10pt]{article}
\usepackage{indentfirst}                              
\usepackage[colorlinks]{hyperref}
\hypersetup{linkcolor=black,filecolor=black,urlcolor=black, citecolor=black}  
\usepackage[numbers,sort&compress]{natbib} 
\usepackage[perpage,symbol*]{footmisc}  

\usepackage{amsmath} 
\usepackage{amsfonts}
\usepackage{amssymb} 
\usepackage{bm}           
\usepackage{mathrsfs} 
\usepackage{amsthm}  
\usepackage{accents}  

\usepackage[titletoc,title]{appendix}
\usepackage{titlesec}
\titleformat{\section}{\large\bfseries}{\thesection}{1em}{}
\titleformat{\subsection}[runin]{\bfseries}{\thesubsection.}{0.5em}{}[.]
\titleformat{\subsubsection}[runin]{\bfseries}{\thesubsubsection.}{0.4em}{}[.]

\DeclareMathOperator{\dive}{div}
\DeclareMathOperator{\supp}{supp}

\def\d{\,\mathrm{d}}
\def\p{\partial}

\makeatletter
\def\wideubar{\underaccent{{\cc@style\underline{\mskip10mu}}}}
\def\Wideubar{\underaccent{{\cc@style\underline{\mskip14mu}}}}
\makeatother
\makeatletter
\def\widebar{\accentset{{\cc@style\underline{\mskip10mu}}}}
\def\Widebar{\accentset{{\cc@style\underline{\mskip8mu}}}}
\makeatother

\newcommand{\VERTiii}[1]{{\left\vert\kern-0.3ex\left\vert\kern-0.3ex\left\vert #1
		\right\vert\kern-0.3ex\right\vert\kern-0.3ex\right\vert}}
\newcommand{\VERT}{\vert\kern-0.3ex\vert\kern-0.3ex\vert}
\newcommand{\VERTl}{\left\vert\kern-0.3ex\left\vert\kern-0.3ex\left\vert}
\newcommand{\VERTr}{\right\vert\kern-0.3ex\right\vert\kern-0.3ex\right\vert}
\newcommand{\VERTbig}{\big\vert\kern-0.3ex\big\vert\kern-0.3ex\big\vert}
\newcommand{\VERTBig}{\Big\vert\kern-0.3ex\Big\vert\kern-0.3ex\Big\vert}

\makeatletter
\DeclareFontFamily{OMX}{MnSymbolE}{}
\DeclareSymbolFont{MnLargeSymbols}{OMX}{MnSymbolE}{m}{n}
\SetSymbolFont{MnLargeSymbols}{bold}{OMX}{MnSymbolE}{b}{n}
\DeclareFontShape{OMX}{MnSymbolE}{m}{n}{
	<-6>  MnSymbolE5 <6-7>  MnSymbolE6 <7-8>  MnSymbolE7 <8-9>  MnSymbolE8 <9-10> MnSymbolE9 <10-12> MnSymbolE10 <12->   MnSymbolE12
}{}
\DeclareFontShape{OMX}{MnSymbolE}{b}{n}{
	<-6>  MnSymbolE-Bold5 <6-7>  MnSymbolE-Bold6 <7-8>  MnSymbolE-Bold7 <8-9>  MnSymbolE-Bold8 <9-10> MnSymbolE-Bold9 <10-12> MnSymbolE-Bold10 <12->   MnSymbolE-Bold12
}{}

\let\llangle\@undefined
\let\rrangle\@undefined
\DeclareMathDelimiter{\llangle}{\mathopen}%
{MnLargeSymbols}{'164}{MnLargeSymbols}{'164}
\DeclareMathDelimiter{\rrangle}{\mathclose}%
{MnLargeSymbols}{'171}{MnLargeSymbols}{'171}
\makeatother

\numberwithin{equation}{section}

\newtheorem{lemma}{Lemma}[section]
 \newtheorem{proposition}[lemma]{Proposition}
\newtheorem{theorem}{Theorem}[section]

\newtheorem{definition}{Definition}[section]
\newtheorem{remark}{Remark}[section]

\begin{document}
	\title{{\bf 
	Two-Dimensional Vortex Sheets for the Nonisentropic  Euler Equations: Nonlinear Stability}}
	\author{
		{\sc Alessandro Morando}\thanks{e-mail: alessandro.morando@unibs.it}\\
		{\footnotesize DICATAM, University of Brescia, Via Valotti 9, 25133 Brescia, Italy}
		\\[1mm]
		{\sc Paola Trebeschi}\thanks{e-mail: paola.trebeschi@unibs.it}\\
		{\footnotesize DICATAM, University of Brescia, Via Valotti 9, 25133 Brescia, Italy}\\[1mm]
		{\sc Tao Wang}\thanks{e-mail: tao.wang@whu.edu.cn}\\
		{\footnotesize School of Mathematics and Statistics, Wuhan University, Wuhan 430072, China}\\[-1mm]
		{\footnotesize Hubei Key Laboratory of Computational Science, Wuhan University, Wuhan 430072, China}
	}
	
	\date{}

\maketitle
\begin{abstract}
We show the short-time existence and nonlinear stability of vortex sheets for the nonisentropic compressible Euler equations in two spatial dimensions, based on the weakly linear stability result of Morando--Trebeschi (2008) \cite{MT08MR2441089}.
{The missing normal derivatives are compensated through the equations of the linearized vorticity and entropy when deriving higher-order energy estimates.}
The proof of the resolution for this nonlinear problem follows from certain \emph{a priori} tame estimates on the effective linear problem {in the usual Sobolev spaces} and a suitable Nash--Moser iteration scheme.
		
		\vspace{2mm}
		\noindent{\bf Keywords:} Nonisentropic fluid, compressible vortex sheet, characteristic boundary, existence, nonlinear stability, Nash--Moser iteration.
		
		\vspace{2mm}
		\noindent{\bf Mathematics Subject Classification:}  35L65,  
				76N10,  
		35Q35,  
		35R35,  
		76E17  
	\end{abstract}
	
	\tableofcontents
	

\section{Introduction and the Main Result}\label{sct1}


The motion of a perfect polytropic ideal gas in the whole plane $\mathbb{R}^2$ is described by the compressible Euler equations
\begin{align}
\label{E1}
\left\{\begin{aligned}
&(\partial_t +{\bm{u}}\cdot\nabla) p+{\bm\gamma}p\nabla\cdot {\bm{u}}=0,\\
&\rho(\partial_t +{\bm{u}}\cdot\nabla){\bm{u}}+\nabla p=0,\\
&(\partial_t +{\bm{u}}\cdot\nabla )s=0,
\end{aligned}\right.
\end{align}
where the pressure $p=p(t,x)\in \mathbb R$, velocity ${\bm{u}}=(v(t,x),u(t,x))^{\mathsf{T}}\in \mathbb R^2$,  and entropy $s=s(t,x) \in \mathbb R$ are unknown functions of time $t$ and position $x=(x_1,x_2)^{\mathsf{T}}\in  \mathbb R^2$.  The density $\rho$ obeys the constitutive law
\begin{align}\label{polytropic}
\rho=\rho(p,s):=A p^{\frac{1}{{\bm\gamma}}} \mathrm{e}^{-\frac{s}{{\bm\gamma}}},
\end{align}
where $A>0$ is given and ${\bm\gamma}>1$ is the adiabatic exponent of the gas.

According to \citet{L57MR0093653}, a weak solution $(p,{\bm{u}},s)$ of \eqref{E1} that is smooth on either side of a smooth surface $\Gamma(t):=\{x_2=\varphi (t,x_1)\}$
is said to be a {\em vortex sheet} (even called a {\em contact discontinuity}) provided that it is a classical solution to \eqref{E1} on each side of $\Gamma(t)$ and the following Rankine--Hugoniot conditions hold at each point of $\Gamma(t)$:
\begin{equation}
\label{RH1}
\partial_t\varphi={\bm{u}}^+\cdot \nu={\bm{u}}^-\cdot \nu ,\quad\quad p^+=p^-.
\end{equation}
Here $\nu:=(-\partial_{x_1}\varphi,1)^{\mathsf{T}}$ is a spatial normal vector to $\Gamma(t)$.  As usual, ${\bm{u}}^{\pm}$, $p^{\pm}$, and $s^{\pm}$ denote the restrictions of ${\bm{u}}$, $p$, and $s$ to both sides $\{\pm(x_2-\varphi(t,x_1))>0\}$ of $\Gamma(t)$, respectively.
Recall that conditions \eqref{RH1} are derived from the conservative form of the compressible Euler equations.
These conditions yield that the normal velocity and pressure are continuous across $\Gamma(t)$. Hence the only possible jumps displayed by a vortex sheet concern the tangential velocity and entropy.  Remark also that the first two identities in \eqref{RH1} are the eikonal equations:
\begin{equation}\notag 
\partial_t\varphi+\lambda_2(p^+,{\bf
  u}^+,s^+,\partial_{x_1}\varphi)=0,\quad
\partial_t\varphi+\lambda_2(p^-,{\bm{u}}^-,s^-,\partial_{x_1}\varphi)=0,
\end{equation}
where $\lambda_2(p,{\bm{u}},s,\xi):={\bm{u}}\cdot(\xi,-1)^{\mathsf{T}}$ denotes the second
characteristic field of system \eqref{E1}.

We are interested in the structural stability of vortex sheets to nonisentropic compressible Euler equations \eqref{E1} with the initial data being a perturbation of piecewise constant vortex sheets:
\begin{align} \label{CVS1}
(p,{\bm{u}},s)=\begin{cases}
(\bar p_r,\bar v_{r},\bar u_{r},\bar s_r) &{\rm if}\  x_2>\sigma t+nx_1,\\
(\bar p_l,\bar v_{l},\bar u_{l},\bar s_l) &{\rm if}\  x_2<\sigma t+nx_1.
\end{cases}
\end{align}
Here, $\bar p_{r,l}$, $\bar v_{r,l}$, $\bar u_{r,l}$, $\bar s_{r,l}$, $\sigma,$ and  $n$ are constants and satisfy
\begin{equation*}
\sigma+\bar v_rn-\bar u_r=0,\quad \sigma+\bar v_l n-\bar u_l=0,\quad  \bar p_r=\bar p_l=:\bar{p}>0.
\end{equation*}
Upon changing observer if necessary, we may assume without loss of generality
\begin{equation*}
\bar u_r=\bar u_l=\sigma=n=0,\quad \bar v_r=-\bar v_l=\bar{v}>0.
\end{equation*}

In such nonlinear stability problem, interface $\Gamma(t)$ (namely, function $\varphi$) is a part of unknowns. The usual approach consists of straightening unknown interface $\Gamma(t)$ by means of a suitable change of coordinates in $\mathbb{R}^3$, in order to reformulate the free boundary problem in a fixed domain. Precisely, unknowns $(p,{\bm{u}},s)$ are replaced by functions
\begin{equation*}
(p_{\sharp}^{\pm},{\bm{u}}_{\sharp}^{\pm},s_{\sharp}^{\pm})(t,x_1,x_2) :=(p^{\pm},{\bm{u}}^{\pm},s^{\pm})(t,x_1,\Phi^{\pm}(t,x_1,x_2)) ,
\end{equation*}
where $\Phi^{\pm}$ are smooth functions satisfying
\begin{equation}\label{Phi.a}
\Phi^{\pm}(t,x_1,0)=\varphi(t,x_1)\quad \mathrm{and}\quad  \pm\partial_{x_2}\Phi(t,x_1,x_2)\geq \kappa>0\quad \mathrm{if}\ x_2\geq 0.
\end{equation}
Hereafter we drop the ``$\sharp$'' index and set $U:=(p,v,u,s)^{\mathsf{T}}$ for convenience.
Then the construction of vortex sheets for system \eqref{E1} amounts to proving the existence of smooth solutions $(U^{\pm},\Phi^{\pm})$ to the following initial-boundary value problem:
\begin{subequations}   \label{E0}
	\begin{alignat}{2} \label{E0.a}
	&\mathbb{L}(U^{+},\Phi^{+}):=L(U^+,\Phi^+)U^+ =0&\qquad &\mathrm{if}\  x_2>0,\\\label{E0.a1}
	&\mathbb{L}(U^{-},\Phi^{-}):=L(U^-,\Phi^-)U^- =0&\qquad &\mathrm{if}\  x_2>0,\\ \label{E0.b}
	&\mathbb{B}(U^{+},U^{-},\varphi)=0& \qquad  &\mathrm{if}\ x_2=0,\\  \label{E0.c}
	&(U^{\pm},\varphi)|_{t=0}=(U^{\pm}_0,\varphi_0),& \qquad &
	\end{alignat}
\end{subequations}
where $\mathbb{B}$ denotes the boundary operator
\begin{align}\label{B.def}
\mathbb B(U^+ ,U^-,\varphi):=
\begin{pmatrix}
(v^+-v^-)\vert_{x_2=0} \partial_{x_1}\varphi -(u^+-u^-)\vert_{x_2=0}\\
\partial_t \varphi +v^+\vert _{x_2=0} \partial_{x_1} \varphi -u^+\vert_{x_2=0}\\
(p^+-p^-)\vert_{x_2=0}
\end{pmatrix}.
\end{align}
Owing to transformation \eqref{Phi.a}, differential operator $L(U,\Phi)$  takes the form
\begin{align}\label{L.def}
L(U,\Phi):=I_4\partial_t+A_1(U) \partial_{x_1}+\widetilde{A}_2(U,\Phi)\partial_{x_2},
\end{align}
where $I_4$ is the $4\times4$ identity matrix,
\begin{align*}
\widetilde{A}_2(U,\Phi):=\frac1{\partial_{x_2}\Phi}
\left(A_2(U)-\partial_t\Phi I_4-\partial_{x_1}\Phi A_1(U)\right),
\end{align*}
and
\setlength{\arraycolsep}{4pt}
\begin{align}\notag 
&A_1(U):=\begin{pmatrix}
v & {\bm\gamma}p & 0 & 0\\
1/\rho & v & 0 & 0\\
0 & 0 & v & 0\\
0 & 0 & 0 & v
\end{pmatrix},
\quad A_2(U):=\begin{pmatrix}
u & 0 &{\bm\gamma} p & 0\\
0 & u & 0 & 0\\
1/\rho& 0 & u & 0\\
0 & 0 & 0 & u
\end{pmatrix}.
\end{align}

Observe that equations \eqref{Phi.a}--\eqref{E0} are not enough to determine functions $\Phi^{\pm}$.
Majda's pioneering work in \cite{M83aMR683422, M83bMR699241} specifies $\Phi^{\pm}(t,x_1,x_2)=\pm x_2+\varphi(t,x_1)$, which turns out to be appropriate in the study of shock waves. However, this particular choice does not work well for vortex sheets where boundary $\{x_2=0\}$ becomes characteristic. As in Francheteau--M\'{e}tivier \cite{FM00MR1787068}, we require functions $\Phi^{\pm}$ to satisfy the following eikonal equations:
\begin{equation}\label{Phi.b}
\partial_t\Phi^{\pm}+\lambda_2(p^{\pm},{\bm{u}}^{\pm},s^{\pm},\partial_{x_1}\Phi^{\pm})=0\qquad \mathrm{if}\ x_2\geq 0.
\end{equation}
This choice of $\Phi^{\pm}$  has the advantage to simplify much the expression of  equations \eqref{E0.a}--\eqref{E0.a1}. More importantly,  
the rank of the boundary matrix for problem \eqref{E0}  keeps constant on the whole domain $\{x_2\geq 0\}$, which allows the application of the Kreiss symmetrizer technique to problem \eqref{E0} in the spirit of Majda--Osher \cite{MO75MR0410107}.

\begin{remark}\label{rmk_coupling}
{\rm It is worthwhile noticing that in interior equations \eqref{E0.a}--\eqref{E0.a1}, the ``$+$'' and ``$-$'' states are decoupled, whereas in boundary conditions \eqref{E0.b} the coupling of the two states occurs, \emph{cf.}\;\eqref{B.def}.}
\end{remark}


In the new variables, piecewise constant state \eqref{CVS1} corresponds to the following trivial solution of \eqref{Phi.a}--\eqref{E0.b} and \eqref{Phi.b}
\begin{equation}\label{CVS0}
\widebar{U}^{\pm}=(\bar{p}, \pm \bar{v}, 0, \bar{s}^{\pm})^{\mathsf{T}},\quad
\widebar{\Phi}^{\pm}(t,x_1,x_2)=\pm x_2,
\end{equation}
with $\bar{p}>0$ and $\bar{v}>0$.
Let us denote by $\bar{c}_\pm=c(\bar{p},\bar{s}^{\pm})$ the sound speeds corresponding to the constant states $\widebar{U}^\pm$,
where for the polytropic gas \eqref{polytropic},
\begin{align*}
c(p,s):=\sqrt{p_{\rho}(\rho,s)}=\sqrt{\frac{\gamma e^{s/\gamma}}{A p^{\frac{1}{{\bm\gamma}}-1}}}\,.
\end{align*}

We aim to show the short-time existence of solutions to nonlinear problem \eqref{Phi.a}--\eqref{E0} and \eqref{Phi.b} provided the initial data is sufficiently close to \eqref{CVS0}. Our main result is stated as follows.
\begin{theorem}\label{thm}
	Let $T>0$  and $\mu\in\mathbb N$ with $\mu\ge 13$. Assume that background state \eqref{CVS0} satisfies the stability conditions
	\begin{equation}\label{H1}
	2\bar{v}>(\bar{c}_+^{\frac23}+\bar{c}_-^{\frac23})^{\frac32},\quad 2\bar{v}\neq\sqrt{2}(\bar{c}_++\bar{c}_-).
	\end{equation}
	Assume further that  the initial data $U^{\pm}_0$ and $\varphi_0$ satisfy the compatibility conditions	up to order $\mu$ in the sense of Definition \ref{def:CC}, and that $(U^{\pm}_0-\widebar{U}^{\pm},\varphi_0)\in H^{\mu+1/2}(\mathbb{R}^2_+)\times H^{\mu+1}(\mathbb{R})$ has a compact support.
	Then there exists $\delta>0$ such that, if $\|U^{\pm}_0-\widebar{U}^{\pm}\|_{H^{\mu+1/2}(\mathbb{R}^2_+)}+\|\varphi_0\|_{H^{\mu+1}(\mathbb{R})}\leq\delta$, then there exists a solution $(U^\pm, \Phi^\pm, \varphi)$ of \eqref{Phi.a}--\eqref{E0} and \eqref{Phi.b} on the time interval $[0,T]$ satisfying
$$
(U^{\pm}-\widebar{U}^{\pm},\Phi^{\pm}-\widebar{\Phi}^{\pm})\in H^{\mu-7}((0,T)\times\mathbb{R}^2_+),\qquad
\varphi\in H^{\mu-6}((0,T)\times\mathbb{R}).
$$
\end{theorem}

Compressible vortex sheets, along with shocks and rarefaction waves, are fundamental waves that  play an important role in the study of general entropy solutions to multidimensional hyperbolic systems of conservation laws. The stability of one uniformly stable shock and multidimensional rarefaction waves has been obtained in \cite{M83aMR683422,M83bMR699241} and \citet{A89MR976971}.

It was observed  long time ago in \citet{M58MR0097930} (\emph{cf}.\;Coulombel--Morando \cite{CM04MR2159807} for using only algebraic tools) that for  two-dimensional nonisentropic Euler equations \eqref{E1}, piecewise constant vortex sheets \eqref{CVS0} are violently unstable unless the following stability criterion is satisfied:
\begin{align} \label{H2}
2\bar{v}\geq (\bar{c}_+^{\frac23}+\bar{c}_-^{\frac23})^{\frac32},
\end{align}
while vortex sheets \eqref{CVS0}  are linearly stable under condition \eqref{H2}.
In the seminal  work of Coulombel--Secchi \cite{CS08MR2423311}, building on their linear stability results in \cite{CS04MR2095445} and a modified Nash--Moser iteration,  the short-time existence and nonlinear stability of  compressible vortex sheets are established for the two-dimensional  {\it isentropic} case under condition \eqref{H2} (as a strict inequality). These results have been generalized very recently by Chen--Secchi--Wang \cite{CSW17Preprint} to cover the relativistic case.

As for three-dimensional gas dynamics, vortex sheets are showed in Fejer--Miles \cite{FM63MR0154509} to be always violently unstable, which is analogous to the Kelvin--Helmholtz instability for incompressible fluids. In contrast, Chen--Wang  \cite{CW08MR2372810} and \citet{T09MR2481071}  proved independently the nonlinear stability of compressible current-vortex sheets for three-dimensional compressible magnetohydrodynamics (MHD). This result indicates that non-paralleled magnetic fields stabilize the motion of three-dimensional compressible vortex sheets.

Extending the results in \cite{CS04MR2095445}, the first two authors in \cite{MT08MR2441089} show the $L^2$--estimates for the linearized problems of \eqref{Phi.a}--\eqref{E0} and \eqref{Phi.b} around background state \eqref{CVS0} under condition \eqref{H2} (as a strict inequality), and that around a small perturbation of \eqref{CVS0} under \eqref{H1}. The main goal of the present paper is to prove the structurally nonlinear stability of two-dimensional nonisentropic vortex sheets  by adopting the Nash--Moser iteration scheme developed in \cite{H76MR0602181,CS08MR2423311} ,which has been successfully applied to the plasma-vacuum interface problem \cite{ST14MR3151094}, three-dimensional compressible steady flows \cite{WY15MR3328144}, and MHD contact discontinuities \cite{MTT18MR3766987}.
{To work in the usual Sobolev spaces, we compensate the missing normal derivatives through the equations for the linearized vorticity and entropy when deriving the higher-order energy estimates.}

It is worth noting that in the statement of Theorem \ref{thm}, the inequality
\begin{equation}\label{double}
2\bar{v}\neq\sqrt{2}(\bar{c}_++\bar{c}_-)
\end{equation}
is required in addition to stability condition \eqref{H2} (with strict inequality).
This is due to the fact that the linearized problem about piece-wise constant basic state \eqref{CVS0} with $\bar{v}$ taking the critical value  {in} \eqref{double} satisfies an \emph{a priori} estimate with additional loss of regularity from the data, which is related to the presence of a double root of the associated Lopatinski\u{\i} determinant (\emph{cf.}\;\cite[Theorem 3.1]{MT08MR2441089}). At the subsequent level of the variable coefficient linearized problem about a perturbation of \eqref{CVS0}, the authors in \cite{MT08MR2441089} were not able to handle this further loss of regularity, thus the case of $\bar{v}=(\bar{c}_++\bar{c}_-)/\sqrt 2$ is still open.
Notice also that in the isentropic case (where $\bar{c}_+=\bar{c}_-=\bar{c}$), value $(\bar{c}_+^{\frac23}+\bar{c}_-^{\frac23})^{\frac32}$ coincides with $\sqrt{2} (\bar{c}_++\bar{c}_-) $ and condition \eqref{H1} reduces to the supersonic condition $\bar{v}>\sqrt 2\bar{c}$ studied in Coulombel--Secchi \cite{CS08MR2423311}.

The plan of this paper is as follows. In \S\,\ref{sct:ELP}, we introduce the effective linear problem and its reformulation. Section \ref{sct:well} is devoted to the proof of a well-posedness result of the effective linear problem in the usual Sobolev space $H^s$ with $s$ large enough. In particular, the weighted Sobolev spaces and norms are introduced in \S\,\ref{sct:weighted} for the sake of completeness.
In \S\,\ref{sct:AS},  by imposing necessary  {compatibility} conditions on the initial data, we introduce the smooth ``approximate solution'', which reduces problem \eqref{Phi.a}--\eqref{E0} and \eqref{Phi.b} into a nonlinear one with zero initial data. In \S\,\ref{sct:N-M}, we use a modification of the Nash--Moser iteration scheme to derive the existence of solutions to the reduced problem and conclude the proof of our main result, namely Theorem \ref{thm}.

\section{The Effective Linear Problem and Reformulation}\label{sct:ELP}
In this section,  we employ the so-called ``good unknowns'' of \citet{A89MR976971} to obtain the effective linear problem and transform it into an equivalent problem with a  constant and diagonal boundary matrix.
\subsection{The effective linear problem}
We are going to linearize the nonlinear problem \eqref{Phi.a}--\eqref{E0} and \eqref{Phi.b} around a basic state $(U_{r,l},\Phi_{r,l}):=(p_{r,l},v_{r,l},u_{r,l},s_{r,l},\Phi_{r,l})^{\mathsf{T}}$ given by a perturbation of the stationary solution \eqref{CVS0}. The index $r$ (resp.\;$l$) denotes the state on the right (resp.\;on the left) of the interface (after change of variables). More precisely, the perturbation
\begin{equation}\label{eq_state}
(\dot{U}_{r,l}(t,x_1,x_2),\dot{\Phi}_{r,l}(t,x_1,x_2)):=(U_{r,l}(t,x_1,x_2),\Phi_{r,l}(t,x_1,x_2))-(\widebar{U}^\pm,\widebar{\Phi}^\pm)
\end{equation}
is assumed to satisfy
\begin{align}
&\mathrm{supp}\,\big(\dot{U}_{r,l},\dot{\Phi}_{r,l}\big)\subset \{-T\le t\le 2T,\ x_2\geq 0,\ |x|\leq R\},\label{bas.c1}\\ \label{bas.c2}
&\dot{U}_{r,l}\in W^{2,\infty}(\Omega),\quad  \dot{\Phi}_{r,l}\in W^{3,\infty}(\Omega),\quad
\big\|\dot{U}_{r,l}\big\|_{W^{2,\infty}(\Omega)}+\big\|\dot{\Phi}_{r,l}\big\|_{W^{3,\infty}(\Omega)}\leq K,
\end{align}
where $T$, $R$, and $K$ are positive constants and $\Omega$ denotes the half-space $\{(t,x_1,x_2)\in\mathbb{R}^3:x_2>0\}$.
Moreover, we assume that $(\dot{U}_{r,l},\dot{\Phi}_{r,l})$ satisfies constraints \eqref{Phi.a}, \eqref{Phi.b}, and Rankine--Hugoniot conditions \eqref{E0.b}, that is,
\begin{subequations}\label{bas.eq}
	\begin{alignat}{2}
	\label{bas.eq.1}&\partial_t \Phi_{r,l}+v_{r,l}\partial_{x_1}\Phi_{r,l}-u_{r,l}=0&\qquad &\mathrm{if}\ x_2\geq 0,\\
	\label{bas.eq.2}&\pm\partial_{x_2}{\Phi}_{r,l}\geq \kappa_0>0&\qquad &\mathrm{if}\ x_2\geq 0,\\
	\label{bas.eq.3}&{\Phi}_{r}={\Phi}_{l}=\varphi &\qquad &\mathrm{if}\ x_2= 0,\\
	\label{bas.eq.4}&\mathbb{B}\big({U}_r,{U}_l,{\varphi}\big)=0&\qquad &\mathrm{if}\ x_2= 0,
	\end{alignat}
\end{subequations}
for a suitable positive constant $\kappa_0$.

The linearized problem of \eqref{Phi.a}--\eqref{E0} and \eqref{Phi.b} around the basic state $(U_{r,l},\Phi_{r,l})$ is given by
\begin{align*}
\left\{\begin{aligned}
&\mathbb L^\prime(U_{r,l},\Phi_{r,l})(V^\pm,\Psi^\pm):=\left.\frac{\mathrm{d}}{\mathrm{d}\theta}\mathbb L(U_{r,l}+\theta V^\pm,\Phi_{r,l}+\theta\Psi^\pm)\right|_{\theta=0}=f^\pm,&&\quad x_2>0,\\
&\mathbb B^\prime(U_{r,l}, \Phi_{r,l})(V^+ ,V^- ,\psi) :=\left.\frac{\mathrm{d}}{\mathrm{d}\theta}\mathbb B(U_r+\theta V^+,U_l+\theta V^-,\varphi+\theta \psi)\right|_{\theta=0}=g,&&\quad x_2=0,
\end{aligned}\right.
\end{align*}
with given source terms $f^\pm=f^\pm(t,x)$ and boundary data $g=g(t,x_1)$, where $\psi$ denotes the common trace of $\Psi^\pm$ on the boundary $\{x_2=0\}$.

In order to get rid of the first order terms in $\Psi^\pm$ involved in the expression of linear operators $\mathbb L^\prime(U_{r,l},\Phi_{r,l})(V^\pm,\Psi^\pm)$, we utilize the ``good unknowns'' of \citet{A89MR976971}:
\begin{equation}\label{Alinhac}
\dot V^+ := V^+ -\frac{\Psi^+}{\partial_{x_2}\Phi_r} \partial_{x_2} U_r,\quad
\dot V^- := V^- -\frac{\Psi^-}{\partial_{x_2}\Phi_l} \partial_{x_2} U_l,
\end{equation}
to get (\emph{cf.}\;\citet[Proposition\,1.3.1]{M01MR1842775})
\begin{equation}\label{L.prime}
\begin{split}
&\mathbb L^\prime(U_{r,l},\Phi_{r,l})(V^\pm,\Psi^\pm) =L(U_{r,l},\Phi_{r,l})\dot V^\pm+\mathcal{C}(U_{r,l},\Phi_{r,l})\dot V^\pm+\frac{\Psi^\pm}{\partial_{x_2}\Phi_{r,l}}\partial_{x_2}\left\{\mathbb L(U_{r,l},\Phi_{r,l})\right\},
\end{split}
\end{equation}
where differential operators $L(U_{r,l},\Phi_{r,l})$ are defined in \eqref{L.def}, while
\begin{equation}\label{C.cal}
\mathcal{C}(U_{r,l},\Phi_{r,l})X:=\mathrm{d} A_1(U_{r,l})X\,\partial_{x_1}U+\big\{\mathrm{d} A_2(U_{r,l})X-\partial_{x_1}\Phi_{r,l}\mathrm{d} A_1(U_{r,l})X\big\}\frac{\partial_{x_2}U_{r,l}}{\partial_{x_2}\Phi_{r,l}},
\end{equation}
for all $X\in\mathbb R^4$. Notice that matrices $\mathcal{C}(U_{r,l},\Phi_{r,l})$ are $C^{\infty}$--functions
of $(\dot U_{r,l},\nabla \dot{U}_{r,l},\nabla\dot\Phi_{r,l})$ that vanish at the origin.

Let us denote $\dot V:=(\dot V^+,\dot V^-)^{\mathsf{T}}$, $\nabla_{t,x_1} \psi=(\partial_t \psi,
\partial_{x_1} \psi)^{\mathsf{T}}$. In terms of new unknowns \eqref{Alinhac}, a direct computation yields
\begin{align}\label{B.prime}
\mathbb B^\prime(U_{r,l}, \Phi_{r,l})(\dot V\vert_{x_2=0},\psi) =
\underline{b} \nabla_{t,x_1} \psi +\bm{b}_{\sharp}\psi+\Wideubar{M}  \dot{V}|_{_{x_2=0}}.
\end{align}
Coefficients $\underline{b}$, $\bm{b}_{\sharp}$, and $\Wideubar{M}$ are defined by
\begin{align}\label{b.bar}
&\underline{b} (t,x_1) := \begin{pmatrix}
0 & (v_r -v_l)|_{x_2=0} \\
1 & {v_r}|_{x_2=0}\\
0 & 0 \end{pmatrix},\quad
\bm{b}_{\sharp} (t,x_1) := \Wideubar{M}(t,x_1)\left.\begin{pmatrix}\displaystyle\frac{\partial_{x_2} U_r}{\partial_{x_2}\Phi_r}\\[3mm] \displaystyle\frac{\partial_{x_2} U_l}{\partial_{x_2}\Phi_l}\end{pmatrix}\right|_{x_2=0},\\
&\Wideubar{M}(t,x_1) :=\begin{pmatrix}
0 & \partial_{x_1}\varphi & -1 & 0 & 0 & -\partial_{x_1}\varphi & 1 & 0\\
0 & \partial_{x_1}\varphi & -1 & 0 & 0 & 0 & 0 & 0\\
1 & 0 & 0 & 0 & -1 & 0 & 0 & 0 \end{pmatrix}. \label{M.bar}
\end{align}
We note that $\bm{b}_{\sharp}$ is a $C^{\infty}$--function of $(\partial_{x_2}\dot{U}_{r,l}|_{x_2=0},  \partial_{x_1}\varphi, \partial_{x_2}\dot\Phi_{r,l}|_{x_2=0})$ vanishing at the origin.

In view of the nonlinear results obtained in \cite{A89MR976971,FM00MR1787068,CS08MR2423311}, we neglect the last term in \eqref{L.prime}, which is the zero-th order term in $\Psi^{\pm}$, and consider the following \emph{effective linear problem}:
\begin{subequations} \label{ELP}
	\begin{alignat}{3}   \label{ELP.1}
	&\mathbb L^\prime_{e}(U_{r},\Phi_{r})\dot{V}^{+} := L(U_{r},  \Phi_{r}) \dot{V}^{+}
	+\mathcal{C}(U_{r},\Phi_{r}) \dot{V}^{+}=f^{+}&\quad  &  \textrm{if } x_2>0,\\ \label{ELP.1.1}
	&\mathbb L^\prime_{e}(U_{l},\Phi_{l})\dot{V}^{-} := L(U_{l},  \Phi_{l}) \dot{V}^{-}
	+\mathcal{C}(U_{l},\Phi_{l}) \dot{V}^{-}=f^{-}&\quad  & \textrm{if } x_2>0,\\
	&\mathbb B'_e(U_{r,l}, \Phi_{r,l})(\dot{V},\psi) :=
	\underline{b} \nabla_{t,x_1} \psi +\bm{b}_{\sharp}\psi+\Wideubar{M}  \dot{V}|_{_{x_2=0}} =g&\quad  &  \textrm{if } x_2=0, \label{ELP.2}\\
	&\Psi^+=\Psi^-=\psi&\quad  &\textrm{if } x_2=0\,, \label{ELP.3}
	\end{alignat}
\end{subequations}
where $\mathbb B^\prime_e\equiv\mathbb B^\prime$ defined in \eqref{B.prime} (differently from interior equations \eqref{ELP.1}--\eqref{ELP.1.1}, in linearized boundary conditions \eqref{ELP.2} all terms involved in \eqref{B.prime} are considered).
It follows from \eqref{bas.c2} that $\underline{b}, \Wideubar{M}\in W^{2,\infty}(\mathbb R^2)$, $\bm{b}_{\sharp}\in W^{1,\infty}(\mathbb R^2)$, $\mathcal{C}(U_{r,l},\Phi_{r,l}) \in W^{1,\infty}(\Omega)$, and the coefficients of the operators $L(U_{r,l},  \Phi_{r,l}) $ are in $W^{2,\infty}(\Omega)$. We will consider the dropped terms in \eqref{ELP.1}--\eqref{ELP.1.1} as error terms at each Nash--Moser iteration step in the subsequent nonlinear analysis.

We observe that linearized boundary conditions \eqref{ELP.2} depend on the traces of $\dot{V}^\pm$ only through the vector $\mathbb{P}(\varphi){\dot{V}}:=(\mathbb{P}(\varphi){\dot{V}}^+,\mathbb{P}(\varphi){\dot{V}}^-)^{\mathsf{T}}$ with
\begin{equation}
\label{P.bb}
\mathbb{P}(\varphi){\dot{V}}^\pm:= ( \dot{V}^\pm_{1} ,\, \dot{V}^\pm_{3}-\partial_{x_1} \varphi \, \dot{V}^\pm_{2})^{\mathsf{T}}.
\end{equation}
In view of \eqref{bas.eq.1}, the coefficients for the normal derivative $\partial_{x_2}$ in operators $\mathbb L^\prime_e(U_{r,l},\Phi_{r,l})$ take the form:
\begin{equation}\label{A2tilde2}
\begin{split}
\widetilde{A}_2(U_{r,l},\Phi_{r,l})=\frac1{\partial_{x_2}\Phi_{r,l}}\begin{pmatrix}0 & -{\bm\gamma} p_{r,l}\partial_{x_1}\Phi_{r,l} & {\bm\gamma} p_{r,l} & 0\\
-\partial_{x_1}\Phi_{r,l}/\rho_{r,l}& 0 & 0 & 0\\
1/\rho_{r,l} & 0 & 0 & 0\\
0 & 0 & 0 & 0\end{pmatrix},
\end{split}
\end{equation}
with $\rho_{r,l}=\rho(p_{r,l}, s_{r,l})$. Hence the boundary $\{x_2=0\}$ is characteristic for problem \eqref{ELP} and we expect  to control only the traces of the {\it noncharacteristic components} of unknown $\dot V^\pm$, namely the vector $\mathbb{P}(\varphi){\dot{V}}^\pm$.

\subsection{Reformulation}
It is more convenient to rewrite the effective linear problem \eqref{ELP} into an equivalent problem with a constant and diagonal boundary matrix.
This can be achieved because the boundary matrix for \eqref{ELP} has constant rank on the whole closed half-space $\{x_2\geq 0\}$.

Let us consider the coefficient matrices of $\partial_{x_2}\dot{V}^\pm$  in \eqref{ELP.1}--\eqref{ELP.1.1} (\emph{cf.}\;\eqref{A2tilde2}). It is easily verified that the eigenvalues are
\begin{equation*}
\lambda^*=0\quad{\rm with\ multiplicity}\ 2,\quad
\lambda_{2}=\frac{c(p,s)}{\partial_{x_2}\Phi}\langle\partial_{x_1}\Phi\rangle,\quad
\lambda_{3}=-\frac{c(p,s)}{\partial_{x_2}\Phi}\langle\partial_{x_1}\Phi\rangle,
\end{equation*}
where the notation $\langle\partial_{x_1}\Phi\rangle =
\sqrt{1+(\partial_{x_1}\Phi)^2}$ has been used and we drop the subscripts $r,l$ from the basic state $(U_{r,l},\Phi_{r,l})$ for simplicity. Define
\begin{equation}\label{T}
T(U,\Phi):=\begin{pmatrix}
0 & \langle\partial_{x_1}\Phi\rangle &
\langle\partial_{x_1}\Phi\rangle & 0 \\[2mm]
1 & -\displaystyle\frac{c}{{\bm\gamma} p}\partial_{x_1}\Phi & \displaystyle\frac{c}{\gamma
	p}\partial_{x_1}\Phi & 0\\[3mm]
\partial_{x_1}\Phi & \displaystyle\frac{c}{{\bm\gamma}p} & -\displaystyle\frac{c}{{\bm\gamma}p} & 0\\[3mm]
0 & 0 & 0 & 1
\end{pmatrix},
\end{equation}
and $A_0 (U, \Phi)=\mathrm{diag}\,(1,\lambda_2^{-1},\lambda_3^{-1},1)$.
Then  it follows
\begin{equation*}
A_0(U,\Phi)T^{-1}(U,\Phi)\widetilde A_2(U,\Phi)T(U,\Phi)={\bm{I}}_2 :=\mathrm{diag}\, (0,1,1,0).
\end{equation*}
In view of this identity, we perform the transformation:
\begin{equation}\label{W.def}
W^+:=T^{-1}(U_r,\Phi_r)\, \dot V^+ ,\quad
W^-:=T^{-1}(U_l,\Phi_l)\, \dot V^- ,
\end{equation}
and set for brevity
\begin{equation*}
T_{r,l} := T (U_{r,l}, \Phi_{r,l})  ,\quad
{\bm{A}}_0^{r,l} := A_0 (U_{r,l},\Phi_{r,l}).
\end{equation*}
Multiplying \eqref{ELP.1}--\eqref{ELP.1.1} by  ${\bm{A}}_0^{r,l} T_{r,l}^{-1}$ yields the equivalent system of \eqref{ELP.1}--\eqref{ELP.1.1}:
\begin{equation}\label{W.eq}
{\bm{A}}_0^{r,l} \, \partial_t W^\pm +{\bm{A}}_1^{r,l} \, \partial_{x_1} W^\pm
+{\bm{I}}_2 \, \partial_{x_2} W^\pm +{\bm{C}}^{r,l} \, W^\pm
=F^\pm,
\end{equation}
where we have set $F^\pm={\bm{A}}_0^{r,l}T_{r,l}^{-1}f^\pm$,  ${\bm{A}}_1^{r,l}:={\bm{A}}_0^{r,l} T^{-1}_{r,l} A_1(U_{r,l}) T_{r,l},$ and
\begin{equation*}
\begin{split}
&{\bm{C}}^{r,l}:={\bm{A}}_0^{r,l}T^{-1}_{r,l}\left[ \partial_tT_{r,l} +  A_1(U_{r,l}) \partial_{x_1} T_{r,l}
+ \widetilde A_2(U_{r,l},\Phi_{r,l})\partial_{x_2} T_{r,l} + \mathcal{ C}(U_{r,l},\Phi_{r,l})T_{r,l}\right].
\end{split}
\end{equation*}
The coefficient matrices ${\bm{A}}^{r,l}_j\in W^{2,\infty}(\Omega)$ ($j=0,1$) and ${\bm{C}}^{r,l}\in W^{1,\infty}(\Omega)$. Moreover, ${\bm{A}}^{r,l}_j$ are $C^{\infty}$--functions of their arguments $(U_{r,l},\nabla\Phi_{r,l})$, and ${\bm{C}}^{r,l}$ are $C^{\infty}$--functions of their arguments  $(U_{r,l},\nabla U_{r,l},\nabla\Phi_{r,l},\nabla^2\Phi_{r,l})$. In terms of the vector $W:=(W^+,W^-)^{\mathsf{T}}$ as defined by \eqref{W.def}, the boundary conditions \eqref{ELP.2}--\eqref{ELP.3} become equivalent to
\begin{subequations}\label{W.bdy}
	\begin{alignat}{3}
	&\mathcal{B}(W,\psi):=
	\underline{b} \, \nabla_{t, x_1} \psi +\bm{b}_{\sharp} \, \psi
	+{\bm{M}}{W}|_{x_2=0} =g\quad  &\mathrm{if}\ x_2=0, \label{W.bdy.1}\\
	&\Psi^+=\Psi^-=\psi\quad  &\mathrm{if}\ x_2=0, \label{W.bdy.2}
	\end{alignat}
\end{subequations}
where $\underline{b}$, $\mathbf{b}_\sharp$ are given by \eqref{b.bar}, and
\begin{equation}\label{M}
{\bm{M}}:=\Wideubar{M}\left. \begin{pmatrix}
T_r & 0 \\
0 & T_l \end{pmatrix}\right|_{x_2=0}.
\end{equation}
It is clear that matrix ${\bm{M}}\in W^{2,\infty}(\mathbb R^2)$ is a $C^{\infty}$--function of $(U_{r,l},\nabla\Phi_{r,l})\vert_{x_2=0}$.

\begin{remark}
{\rm We find that the trace of vector ${\bm{M}}W$ involved in boundary conditions \eqref{W.bdy.1} depends only on the trace of the {\it noncharacteristic part} of vector $W^\pm$, {\it i.e.} sub-vector $W^{\rm nc}:=(W^+_2,W^-_2,W^+_3,W^-_3)^{\mathsf{T}}$. Moreover, from \eqref{W.def} and \eqref{P.bb}, one can compute explicitly
\begin{equation}\label{P.bb2}
\mathbb P(\varphi)\dot V^\pm\vert_{x_2=0}=
\left.\begin{pmatrix}
\langle\partial_{x_1}\varphi\rangle (W^\pm_2+W^\pm_3)\\[2mm] \frac{c_{r,l}}{{\bm\gamma} p}\langle\partial_{x_1}\varphi\rangle^2 (W^\pm_2-W^\pm_3)
\end{pmatrix}\right|_{x_2=0} ,
\end{equation}
where $p:=p_{r,l}\vert_{x_2=0}$ and $c_{r,l}=c(p_{r,l},s_{r,l})$.
In the following we shall feel free to rewrite the product ${\bm{M}}W\vert_{x_2=0}$ as ${\bm{M}}W^{\rm nc}\vert_{x_2=0}$ with a slight abuse of notation.}
 \end{remark}
\section{Well-posedness of the Effective Linear Problem}\label{sct:well}
In this section, we are going to establish a well-posedness result for the effective linear problem \eqref{ELP} in the usual Sobolev space $H^s$ with $s$ large enough. The essential point is to deduce certain {\it a priori tame} estimates in $H^s$.
For a general hyperbolic problem with a characteristic boundary,
there is a loss of control near the boundary of normal derivatives in \emph{a priori} energy estimates,
and hence Sobolev spaces with conormal regularity are required (see \citet{S96MR1405665} and the references therein).
However, for problem \eqref{ELP}, we can manage to compensate the loss of normal derivatives and obtain \emph{a priori} estimates in the usual Sobolev spaces. This is achieved by employing the idea in \cite{CS08MR2423311} and estimating the missing derivatives through the equations of the linearized vorticity and entropy. For all notation used below, we address the reader to Section \ref{sct:weighted}.

The main result in this section is stated as follows:

\begin{theorem}\label{thm:L}
	Let $T>0$ and $s\in [3,\tilde{\alpha}]\cap \mathbb{N}$ with any integer $ \tilde{\alpha}\geq 3$.
	Assume that the stationary solution \eqref{CVS0} satisfies \eqref{H1},
	and that perturbations $(\dot{U}_{r,l},\dot{\Phi}_{r,l})$
	belong to $H^{s+3}_{\gamma}(\Omega_T)$ for all $\gamma\geq 1$ and satisfy \eqref{bas.c1}--\eqref{bas.eq},
	and
	\begin{align} \label{H.L.1}
	\|(\dot{U}_{r,l},\nabla\dot{\Phi}_{r,l})\|_{H^5_{\gamma}(\Omega_T)}
	+\|(\dot{U}_{r,l},\p_{x_2}\dot{U}_{r,l},\nabla\dot{\Phi}_{r,l})|_{x_2=0}\|_{H^4_{\gamma}(\omega_T)}\leq K.
	\end{align}
	Assume further that $(f^\pm,g)\in H^{s+1}(\Omega_T)\times H^{s+1}(\omega_T)$ vanish in  the past.
	Then there exists a positive constant $K_0$, which is independent of $s$ and  $T$, and there exist two constants $C>0$ and $\gamma\geq 1$,
	which depend solely on $K_0$, such that, if $K\leq K_0$, then problem \eqref{ELP}
	admits a unique solution $(\dot{V}^{\pm},\psi)\in H^{s}(\Omega_T)\times H^{s+1}(\omega_T)$ that vanishes in the past
	and obeys the following tame estimate:
	\begin{align} \notag
	&\|\dot{V}\|_{H^{s}_{\gamma}(\Omega_T)}+\|\mathbb{P}({\varphi})\dot{V}\,\!|_{x_2=0}
	\|_{H^{s}_{\gamma}(\omega_T)}+\|\psi\|_{H^{s+1}_{\gamma}(\omega_T)}\\
	&\leq C\big\{ \|f\|_{H^{s+1}_{\gamma}(\Omega_T)}
	+\|g\|_{H^{s+1}_{\gamma}(\omega_T)}+\big(\|f\|_{H^4_{\gamma}(\Omega_T)}
	+\|g\|_{H^4_{\gamma}(\omega_T)}\big)\|(\dot{U}_{r,l},\dot{\Phi}_{r,l})\|_{H^{s+3}_{\gamma}(\Omega_T)} \big\}\,,
	\label{tame}
	\end{align}
	where $\dot V:=(\dot V^+,\dot V^-)$, $\mathbb P(\varphi)\dot V:=(\mathbb P(\varphi)\dot V^+, \mathbb P(\varphi)\dot V^-)$, $f:=(f^+,f^-)$ and we have used definition \eqref{pm_notazione}.
\end{theorem}
We consider the particular case where the source terms $f^\pm$ and $g$ vanish in the past,
which corresponds to the nonlinear problem \eqref{Phi.a}--\eqref{E0} and \eqref{Phi.b} with zero initial data.
The general case is postponed to the nonlinear analysis which involves the construction of
a so-called {\it approximate solution}.

\begin{remark}
{\rm Let us notice that in inequality \eqref{tame} the involved norms are the natural ones in the {exponentially weighted} Sobolev spaces $H^s_\gamma(\Omega_T)$ and $H^s_\gamma(\omega_T)$ for a suitable fixed $\gamma$ (see Section \ref{sct:weighted}), even though the functions considered in the statement of Theorem \ref{thm:L} are taken in the usual Sobolev spaces $H^s(\Omega_T)$, $H^s(\omega_T)$. However the functions involved in the statement above vanish for negative time. This guarantees that they also belong to the weighted Sobolev spaces.
}
\end{remark}

\subsection{Weighted Sobolev spaces and norms} \label{sct:weighted}
To be definite, we introduce certain notations on weighted Sobolev spaces and norms.
For every real number $s$, $H^s(\mathbb{R}^2)$ denotes the usual Sobolev
space of order $s$, which is endowed with the $\gamma-$weighted norm
\begin{equation}\label{gamma_norm}
\Vert u\Vert^2_{s,\gamma}:=\frac{1}{(2\pi)^2} \int_{\mathbb{R}^2}(\gamma^2+|\xi|^2)^{s} |\widehat{u}(\xi)|^2\d \xi,
\end{equation}
where $\gamma\geq 1$ and $\widehat{u}$ is the Fourier transform of any distribution $u$ in $\mathbb{R}^2$.
We will abbreviate the usual norm  of $L^2(\mathbb{R}^2)$ as
$\|\cdot\|:=\|\cdot\|_{0,\gamma} . $

Throughout the paper, we introduce the notation: $A\lesssim B$ ($B\gtrsim A$) if $A\leq CB$ holds uniformly
for some positive constant $C$ that is \emph{independent} of $\gamma$.
The notation, $A\sim B$, means that both $A\lesssim B$ and $B\lesssim A$ are satisfied.
Then, for $k\in\mathbb{N}$, one has
\begin{align}\label{norm_sim}
\|u\|^2_{k,\gamma}\sim \sum_{|\alpha|\leq k}\gamma^{2(k-|\alpha|)}\|\partial^{\alpha} u\|^2
\qquad\,\textrm{for all }u\in H^{k}(\mathbb{R}^2),
\end{align}
where $\partial^\alpha:=\partial^{\alpha_0}_t\partial^{\alpha_1}_{x_1}\partial^{\alpha_2}_{x_2}$ with $\alpha=(\alpha_0,\alpha_1,\alpha_2)\in\mathbb N^3$ and $|\alpha|:=\alpha_0+\alpha_1+\alpha_2.$

For $s\in\mathbb R$ and $\gamma\geq 1$, we denote the $\gamma-$weighted Sobolev space $H^s_{\gamma}(\mathbb{R}^2)$ by
\begin{align}\notag
H^{s}_{\gamma}(\mathbb{R}^2)&:=\left\{
u\in\mathcal{D}'(\mathbb{R}^2)\,:\, \mathrm{e}^{-\gamma t}u(t,x_1)\in
H^{s}(\mathbb{R}^2) \right\},
\end{align}
and its norm by $\|u\|_{H^{s}_{\gamma}(\mathbb{R}^2)}:=\|\mathrm{e}^{-\gamma t}u\|_{s,\gamma}$. We write $L^2_{\gamma}(\mathbb{R}^2):=H^0_{\gamma}(\mathbb{R}^2)$ and $\|u\|_{L^2_{\gamma}(\mathbb{R}^2)}:=\|\mathrm{e}^{-\gamma t}u\|$ for short. We define $L^2(\mathbb{R}_+;H^{s}_{\gamma}(\mathbb{R}^2))$, briefly denoted by $L^2(H^s_{\gamma})$, as the space of distributions with finite $L^2(H^s_{\gamma})$--norm, where
\begin{align*}
\|u\|_{L^2(H^s_{\gamma})}^2:=\int_{\mathbb{R}_+}\|u(\cdot,x_2)\|_{H^s_{\gamma}(\mathbb{R}^2)}^2\d x_2\,.
\end{align*}

Let us recall that $\Omega$ denotes the half-space $\{(t,x_1,x_2)\in\mathbb{R}^3:x_2>0\}$ and let $\mathbb R^2_+:=\{(x_1,x_2)\in\mathbb{R}^2:x_2>0\}$.
The boundaries $\partial\Omega$ and $\partial\mathbb R^2_+$ will be respectively identified to $\mathbb{R}^2$ and $\mathbb R$.
We set $L^2_{\gamma}(\Omega):=L^2(H^0_{\gamma})$ and $\|u\|_{L^2_{\gamma}(\Omega)}$.
For all $k\in\mathbb N$ and $\gamma\geq 1$, the weighted Sobolev space $H^k_{\gamma}(\Omega)$ is defined as
\begin{align*}
H^{k}_{\gamma}(\Omega)&:=\left\{ u\in\mathcal{D}'(\Omega)\,:\, \mathrm{e}^{-\gamma t}u\in   H^{k}(\Omega) \right\}.
\end{align*}

For any real number $T$, we denote $\Omega_T:=(-\infty,T)\times\mathbb R^2_+$ and $\omega_T:=(-\infty,T)\times\mathbb{R}\simeq\partial\Omega_T$.
For all $k\in\mathbb{N}$ and $\gamma\geq 1$, we define the weighted space $H^k_{\gamma}(\Omega_T)$ as
$$
H^k_{\gamma}(\Omega_T):=\left\{u\in\mathcal{D}'(\Omega_T)\,:\, \mathrm{e}^{-\gamma t}u\in H^{k}(\Omega_T)\right\}.
$$
In view of relation \eqref{norm_sim}, the norm on $H^k_{\gamma}(\Omega_T)$ is defined as
\begin{align} \label{norm.def}
\|u\|_{H^k_{\gamma}(\Omega_T)}^2:=\sum_{|\alpha|\leq k}\gamma^{2(k-|\alpha|)}\|\mathrm{e}^{-\gamma t} \partial^{\alpha}u\|_{L^2(\Omega_T)}^2.
\end{align}
The norm on $H^k_{\gamma}(\omega_T)$ is defined in the same way.
For all $k\in\mathbb{N}$ and $\gamma\geq 1$, we define  {the} space $L^2(\mathbb{R}_+;H^k_{\gamma}(\omega_T))$,
briefly denoted by $L^2(H^k_{\gamma}(\omega_T))$, as the space of distributions with
finite $L^2(H^k_{\gamma}(\omega_T))$--norm, where
\begin{align} \label{norm_L2Hk}
\|u\|_{L^2(H^k_{\gamma}(\omega_T))}^2:=\int_{\mathbb{R}_+}\|u(\cdot,x_2)\|_{H^k_{\gamma}(\omega_T)}^2\d x_2
=\sum_{\alpha_0+\alpha_1\leq k}
\gamma^{2(k-\alpha_0-\alpha_1)}\|\mathrm{e}^{-\gamma t} \partial_t^{\alpha_0}\partial_1^{\alpha_1}u\|_{L^2(\Omega_T)}^2.
\end{align}
We write $L^2_{\gamma}(\Omega_T):=L^2(H^0_{\gamma}(\omega_T))$
and $\|u\|_{L^2_{\gamma}(\Omega_T)}:=\|u\|_{L^2(H^0_{\gamma}(\omega_T))}$ for brevity.

\medskip
Since the most of functions we are dealing with have double $\pm$ states, throughout the paper we will make use of the following general shortcut notation: for any function space $\mathcal X$ and every function $W=(W^+,W^{-})$, with scalar/vector $\pm$ states $W^\pm=W^\pm(t,x_1,x_2)$, we set
\begin{equation}\label{pm_notazione}
\Vert W\Vert_{\mathcal X}:=\sum\limits_{\pm}\Vert W^\pm\Vert_{\mathcal X}\,.
\end{equation}

\subsection{Well-posedness in $L^2$}
In this subsection, we apply the well-posedness result in $L^2$ of \citet{C05MR2138641} to the effective linear problem \eqref{ELP}.
We recall that system \eqref{ELP.1}--\eqref{ELP.1.1} is symmetrizable hyperbolic with Friedrichs symmetrizer $\mathcal{s}_{r,l} :=\mathrm{diag}\,(1/\rho_{r,l},{\bm\gamma} p_{r,l},{\bm\gamma} p_{r,l},{\bm\gamma} p_{r,l})$, and observe that the coefficients of the linearized operators
satisfy the regularity assumptions of \cite{C05MR2138641}.

The first two authors prove in \cite[Theorem\;4.1]{MT08MR2441089} that problem \eqref{ELP} satisfies a basic $L^2$--{\it a priori} estimate with a loss of one tangential derivative as follows.
\begin{theorem}[\cite{MT08MR2441089}]\label{thm:MT}
Assume that the stationary solution \eqref{CVS0} satisfies \eqref{H1}, and that the basic state $( U_{r,l}, \Phi_{r,l})$ satisfies \eqref{bas.c1}--\eqref{bas.eq}.
Then there exist constants $K_0>0$ and $\gamma_0\geq 1$ such that,
if $K\leq K_0$ and $\gamma\geq \gamma_0$,
then, for all $(\dot{V}^\pm,\psi)\in H^2_{\gamma}(\Omega)\times H^2_{\gamma}(\mathbb{R}^2)$, the following estimate holds:
\begin{align}
&\gamma \|\dot{V}\|_{L^2_{\gamma}(\Omega)}^2+\|\mathbb{P}({\varphi})\dot{V}|_{x_2=0}\|_{L^2_{\gamma}(\mathbb{R}^2)}^2
+\|\psi\|_{H^1_{\gamma}(\mathbb{R}^2)}^2\notag \\
&\quad \lesssim  \gamma^{-3}\big\|\mathbb L^\prime_e\dot V\big\|_{L^2(H_{\gamma}^1)}^2
+\gamma^{-2}\big\|\mathbb B'_e(U_{r,l}, \Phi_{r,l})(\dot{V}, \psi)\big\|_{H^1_{\gamma}(\mathbb{R}^2)}^2,
\label{E:MT}
\end{align}
where $\mathbb L^\prime_{e}\dot{V}:=(\mathbb L^\prime_{e}(U_{r},\Phi_{r})\dot{V}^{+}, \mathbb L^\prime_{e}(U_{l},\Phi_{l})\dot{V}^{-})$.
\end{theorem}
In view of the result in \cite{C05MR2138641}, we only need to find for \eqref{ELP} a dual problem that obeys an appropriate energy estimate.
To this end, we introduce the following matrices:
\setlength{\arraycolsep}{2pt}
\begin{align}  \label{M1N1}
M_1 &:=\left.\begin{pmatrix}
0 & 0 & 0 & 0 &-{\ell}_1^l & 0 & 0 & 0\\
{\ell}_1^r& 0 & 0 & 0 & {\ell}_1^l & 0 & 0 & 0\\
0 & {\ell}_2^r & {\ell}_3^r & 0 & 0 & -{\ell}_2^l & -{\ell}_3^l & 0
\end{pmatrix}\right|_{x_2=0},\quad
N_1 :=\left.\begin{pmatrix}
0 & 0 & 0 & 0 & 0 & 0 & 0 & 0\\
0 & 0 & 0 & 0 & 0 & 0 & 0 & 0\\
0 &{\ell}_2^r & {\ell}_3^r & 0 & 0 &{\ell}_2^l & {\ell}_3^l & 0
\end{pmatrix}\right|_{x_2=0},
\end{align}
and
\begin{align*}
N:=
\begin{pmatrix}
0 & \partial_{x_1}\varphi & -1 & 0 & 0 & \partial_{x_1}\varphi & -1 & 0\\
0 & \partial_{x_1}\varphi & -1 & 0 & 0 & 0 & 0 & 0\\
1 & 0 & 0 & 0 & 1 & 0 & 0 & 0
\end{pmatrix} ,
\end{align*}
where
\begin{align*}
{\ell}_1^{r,l}:=-\frac{{\bm\gamma} p_{r,l}}{\partial_{x_2}\Phi_{r,l}},\quad
{\ell}_2^{r,l}:=-\frac{\partial_{x_1}\varphi}{2\rho_{r,l}\partial_{x_2}\Phi_{r,l}} ,\quad
{\ell}_3^{r,l}:=\frac{1}{2\rho_{r,l}\partial_{x_2}\Phi_{r,l}} .
\end{align*}
By virtue of \eqref{A2tilde2}, we compute that these matrices satisfy the following relation:
\begin{equation}\label{dual:id1}
\mathrm{diag}\,\big(\widetilde{A}_2({U}_{r},{\Phi}_{r}),
\,\widetilde{A}_2({U}_{l},{\Phi}_{l})\big)
\big|_{x_2=0}=M_1^{\mathsf{T}}\Wideubar{M}+N_1^{\mathsf{T}}N ,
\end{equation}
where matrix $\Wideubar{M}=\Wideubar{M}(t,x_1)$ was defined in \eqref{M.bar}.
Moreover, we infer from \eqref{bas.c2} that all matrices $M_1$,  $N$, $N_1$,
and $\Wideubar{M}$ belong to $W^{2,\infty}(\mathbb{R}^2)$.

Then we define the dual problem for \eqref{ELP}, posed on $\Omega$, as follows:
\begin{align}\label{dual}
\left\{\begin{aligned}
&\mathbb L^{\prime}_{e}(U_{r},\Phi_{r})^*U^{+}=\widetilde{f}_{+},\quad &x_2>0,\\
&\mathbb L^{\prime}_{e}(U_{l},\Phi_{l})^*U^{-}=\widetilde{f}_{-},\quad &x_2>0,\\
&N_1 U=0,\quad &x_2=0,\\
&\dive(\underline b^{\mathsf{T}}M_1  U)-\bm{b}_{\sharp}^{\mathsf{T}} M_1 U=0,\quad &x_2=0,
\end{aligned}\right.
\end{align}
where $\underline b$, $\bm{b}_{\sharp}$,  $M_1$, and  $N_1$ are defined
in \eqref{b.bar}  and \eqref{M1N1}, dual operators $\mathbb L^{\prime}_{e}(U_{r,l},\Phi_{r,l})^\ast$ are the formal adjoints of $\mathbb L^{\prime}_{e}(U_{r,l},\Phi_{r,l})$, and $\dive$ denotes the divergence operator in $\mathbb{R}^2$ with respect to $(t,x_1)$.
We refer to \cite[\S\,3.2]{M01MR1842775} for the derivation of the dual problem
by using identity \eqref{dual:id1} and integration by parts.

Since the rank of $N_1$ is one, the dual problem \eqref{dual} has exactly two independent scalar boundary conditions, which is compatible with the number of the {\it incoming characteristics}, i.e., negative eigenvalues of the boundary matrix for \eqref{dual}.

Assuming \eqref{bas.c1}--\eqref{bas.c2} hold for a sufficiently small $K$,
we can repeat the same analysis as in \cite{MT08MR2441089} to show that the dual problem \eqref{dual} satisfies the backward Lopatinski\u{\i} condition and that the roots of the associated Lopatinski\u{\i} determinant are simple and coincide with those for the original problem \eqref{ELP}.
Then we can  find the desired \emph{a priori} estimate with the loss of one tangential derivative for problem \eqref{dual}.
The effective linear problem \eqref{ELP} thus satisfies all the assumptions ({\it i.e.}, symmetrizability, regularity, and weak stability) listed in \cite{C05MR2138641}. We therefore obtain the following well-posedness result.
\begin{theorem}\label{thm:L1}
Let $T>0$. Assume that all the hypotheses of Theorem \ref{thm:MT} hold.
Then there exist positive constants $K_0>0$ and $\gamma_0\geq 1$, independent of $T$, such that, if $K\leq K_0$,
then, for source terms $f^{\pm}\in L^2(H^1(\omega_T))$ and $g\in H^1(\omega_T)$ that vanish for $t<0$, the problem{\rm :}
	\begin{align*}
	\left\{\begin{aligned}
	&\mathbb{L}'_e\big(U_{r},\Phi_{r}\big)\dot{V}^{+}=f^{+}\,,\quad &\mbox{for $t<T,\ x_2>0$},\\
	&\mathbb{L}'_e\big(U_{l},\Phi_{l}\big)\dot{V}^{-}=f^{-}\,, \quad &\mbox{for $t<T,\ x_2>0$},\\
	&\mathbb{B}'_e\big(U_{r,l}, \Phi_{r,l}\big)(\dot{V}, \psi)=g\qquad  &\mbox{for $t<T,\ x_2=0$},
	\end{aligned}\right.
	\end{align*}
	has a unique solution $(\dot{V}^\pm,\psi)\in L^2(\Omega_T)\times H^1(\omega_T)$
	that vanishes for $t<0$ and satisfies $\mathbb{P}({\varphi})\dot{V}^\pm|_{x_2=0}\in L^2(\omega_T)$.
	Moreover, the following estimate holds for all $\gamma\geq \gamma_0$ and for all $t\in[0,T]${\rm :}
	\begin{align}
	&\gamma\|\dot{V}\|_{L^2_{\gamma}(\Omega_t)}^2
	+\|\mathbb{P}({\varphi})\dot{V} |_{x_2=0}\|_{L^2_{\gamma}(\omega_t)}^2+\|\psi\|_{H^1_{\gamma}(\omega_t)}^2
 \lesssim \gamma^{-3}\|f\|_{L^2(H_{\gamma}^1(\omega_t))}^2+\gamma^{-2}\|g\|_{H^1_{\gamma}(\omega_t)}^2.   \label{E:L1}
	\end{align}
\end{theorem}

Theorem \ref{thm:L1} shows the well-posedness of linearized problem \eqref{ELP} in $L^2$ when source terms $(f^\pm,g)$
belong to $L^2(H^1)\times H^1$. We now turn to the derivation of energy estimates for the higher-order derivatives of solutions.

\subsection{\emph{A priori} tame estimate}
In order to obtain the estimates for higher-order derivatives of solutions to \eqref{ELP},
it is convenient to deal with the reformulated problem \eqref{W.eq} and \eqref{W.bdy} for new unknowns $W$.
Until the end of this section, we always assume that $\gamma\geq \gamma_0$ and $K\leq K_0$,
where $\gamma_0$ and $K_0$ are given by Theorem \ref{thm:L1}.
Then estimate \eqref{E:L1} can be rewritten as
\begin{align}
& \sqrt{\gamma}\|W\|_{L^2_{\gamma}(\Omega_T)}+\|W^{\mathrm{nc}}\,\!|_{x_2=0}\|_{L^2_{\gamma}(\omega_T)}
+\|\psi\|_{H^1_{\gamma}(\omega_T)}\notag\\
&\,\, \lesssim \gamma^{-{3}/{2}}\|F\|_{L^2(H^1_{\gamma}(\omega_T))}
+\gamma^{-1}\|g\|_{H^1_{\gamma}(\omega_T)}. \label{E:L1a}
\end{align}

We first derive the estimate for tangential derivatives.
Let $k\in[1,s]$ be a fixed integer.
Applying the tangential derivative $\p_{\rm tan}^{\alpha}=\p_t^{\alpha_0}\p_{x_1}^{\alpha_1}$ with $|\alpha|=k$
to system \eqref{W.eq} yields the equations for $\p_{\rm tan}^{\alpha}W^{\pm}$ that involve the linear terms
of the derivatives, {\it i.e.}, $\p_{\rm tan}^{\alpha-\beta}\p_t W^{\pm}$ and $\p_{\rm tan}^{\alpha-\beta}\p_{x_1} W^{\pm}$, with $|\beta|=1$.
These terms cannot be treated simply as source terms, owing to the loss of derivatives in energy estimate \eqref{E:L1a}.
To overcome this difficulty, we adopt the idea of \cite{CS08MR2423311}, which is to deal with a boundary value problem for all the tangential derivatives of order equal to $k$, {\it i.e.}, for $W^{(k)}:=\{ \p_t^{\alpha_0}\p_{x_1}^{\alpha_1}W^{\pm},\, \alpha_{0}+\alpha_{1}=k\}$.
Such a problem satisfies the same regularity and stability properties as the original problem \eqref{W.eq} and \eqref{W.bdy}.
Repeating the derivation in \cite[\S\,4]{MT08MR2441089}, we get that $W^{(k)}$ satisfies an energy estimate similar to \eqref{E:L1a} with new source terms $\mathcal{F}^{(k)}$ and $\mathcal{G}^{(k)}$.
Then we can employ the Gagliardo--Nirenberg and Moser-type inequalities ({\it cf}.\;\cite[Theorems 8--10]{CS08MR2423311})
to derive the following estimate for tangential derivatives (see \cite[Proposition 1]{CS08MR2423311} for the detailed proof).
\begin{lemma}[Estimate of tangential derivatives]\label{lem.tame1}
	Assume that the hypotheses of Theorem {\rm \ref{thm:L}} hold.
	Then there exist constants $C_{s}>0$ and $\gamma_{s}\geq 1$,  independent of $T$, such that,
	for all $\gamma\geq \gamma_{s}$ and for all $(W^\pm,\psi)\in H^{s+2}_{\gamma}(\Omega_T)\times H^{s+2}_{\gamma}(\omega_T)$
	that are solutions of problem \eqref{W.eq} and \eqref{W.bdy}, the following estimate holds{\rm :}
	\begin{align}
	&\notag\sqrt{\gamma}\|W\|_{L^2(H^{s}_{\gamma}(\omega_T))}+\|W^{\mathrm{nc}}\,\!|_{x_2=0}\|_{H^{s}_{\gamma}(\omega_T)}
	+\|\psi\|_{H^{s+1}_{\gamma}(\omega_T)}\\
	& \leq C_{s}\big\{\frac{1}{\gamma}\|g\|_{H^{s+1}_{\gamma}(\omega_T)}
	+\frac{1}{\gamma^{{3}/{2}}}\big\|F\big\|_{L^2(H^{s+1}_{\gamma}(\omega_T))}
	+\frac{1}{\gamma^{{3}/{2}}}\|W\|_{W^{1,\infty}(\Omega_T)}\big\|\big(\dot{U}_{r,l},\nabla\dot{\Phi}_{r,l}\big)\big\|_{H^{s+2}_{\gamma}(\Omega_T)}\notag \\
	&\qquad\quad\, +\frac{1}{\gamma} \big(\|W^{\mathrm{nc}}\,\!|_{x_2=0}\|_{L^{\infty}(\omega_T)}
	+\|\psi\|_{W^{1,\infty}(\omega_T)}\big)\big\|\big(\dot{U}_{r,l},\p_{x_2}\dot{U}_{r,l},\nabla_x\dot{\Phi}_{r,l}\big)|_{x_2=0}\big\|_{{H^{s+1}_{\gamma}(\omega_T)}}\big\}.
\notag 
	\end{align}
\end{lemma}

Since the boundary matrix for our problem \eqref{W.eq} and \eqref{W.bdy} is singular, there is no hope to estimate all the normal derivatives of $W$ directly from equations \eqref{W.eq} by applying the standard approach for noncharacteristic boundary problems as in \cite{RM74MR0340832,
	MT13MR3085779}.
However, for our problem \eqref{ELP}, we can obtain the estimate of missing normal derivatives through the {equations of the ``linearized vorticity'' and entropy, where the linearized vorticity} has been introduced in \cite{CS08MR2423311} and defined as:
\begin{equation}\label{xi.def}
\dot{\xi}^\pm:=\partial_{x_1}\dot{u}^\pm-\frac1{\partial_{x_2}\Phi_{r,l}}\left(\partial_{x_1}\Phi_{r,l}\partial_{x_2}\dot{u}^\pm+\partial_{x_2}\dot{v}^\pm\right),
\end{equation}
where  $\dot u^\pm$ and $\dot v^\pm$ are the second and third components of good unknowns $\dot{V}^\pm$ (\emph{cf.}\;\eqref{Alinhac}), respectively.

Plugging \eqref{xi.def} into \eqref{ELP.1}--\eqref{ELP.1.1}, we can compute
\begin{equation}\label{xi.eq}
\left(\partial_t+v_{r,l}\partial_{x_1}\right)\dot\xi^\pm= \left(\partial_{x_1}-\frac{\partial_{x_1}\Phi_{r,l}}{\partial_{x_2}\Phi_{r,l}}\partial_{x_2}\right)
\mathcal{F}_2^\pm
-\frac {\partial_{x_2}\mathcal{F}_1^\pm}{\partial_{x_2}\Phi_{r,l}}+\Lambda_1^{r,l}\partial_{x_1}\dot V^\pm+\Lambda^{r,l}_2\partial_{x_2}\dot V^\pm,
\end{equation}
where $\mathcal{F}_1^{\pm}$ and $\mathcal{F}_2^{\pm}$ are given by
\begin{align} \notag
\mathcal{F}_1^{\pm}:=\big(f^{\pm}-\mathcal{C}({U}_{r,l}, \Phi_{r,l})\dot{V}^{\pm}\big)_2,
\qquad \mathcal{F}_2^{\pm}:=\big(f^{\pm}-\mathcal{C}({U}_{r,l}, \Phi_{r,l})\dot{V}^{\pm}\big)_3,
\end{align}
and vectors $\Lambda^{r,l}_{1,2}$ are $C^{\infty}$--functions of $(\dot{U}_{r,l},\nabla\dot{U}_{r,l},\nabla\dot\Phi_{r,l}, \nabla^2\dot\Phi_{r,l})$ vanishing at the origin.

Applying a standard energy method to the transport equations \eqref{xi.eq},
we can utilize the Gagliardo--Nirenberg and Moser-type inequalities to derive the following estimate of $\dot{\xi}^{\pm}$:
\begin{lemma}[Estimate of vorticity]\label{lem.tame2}
	Assume that the hypotheses of Theorem {\rm \ref{thm:L}} hold.
	Then there exist constants $C_{s}>0$ and $\gamma_{s}\geq 1$, independent of $T$, such that,
	for all $\gamma\geq \gamma_{s}$ and for all $(W^\pm,\psi)\in H^{s+2}_{\gamma}(\Omega_T)\times H^{s+2}_{\gamma}(\omega_T)$
	that are solutions of problem \eqref{W.eq} and \eqref{W.bdy},
	functions $\dot{\xi}^{\pm}$ defined by \eqref{xi.def}
	satisfy the following estimate:
	\begin{align}
	\notag \gamma\big\|\dot{\xi}^{\pm}\big\|_{H^{s-1}_{\gamma}(\Omega_T)}
	\leq  C_{s}\big\{& \big\|f^{\pm}\big\|_{H^{s}_{\gamma}(\Omega_T)}
	+\big\|f^{\pm}\big\|_{L^{\infty}(\Omega_T)}\big\|\nabla\dot\Phi_{r,l}\big\|_{H^{s}_{\gamma}(\Omega_T)}
	+\big\|\dot{V}^{\pm}\big\|_{H^{s}_{\gamma}(\Omega_T)}
	\\ &+\big\|\dot{V}^{\pm}\big\|_{W^{1,\infty}(\Omega_T)}\big(\big\|\dot U_{r,l}\big\|_{H^{s+1}_{\gamma}(\Omega_T)}
	+\big\|\nabla\dot{\Phi}_{r,l}\big\|_{H^{s}_{\gamma}(\Omega_T)}\big)\big\}.
	\notag 
	\end{align}
\end{lemma}
We turn to the estimate for normal derivatives $\p_{x_2}(W^{\pm}_1,W^{\pm}_2,W^{\pm}_3)$ by expressing them in terms of tangential derivatives $\p_t W^{\pm}$, $\p_{x_1} W^{\pm}$, and vorticity $\dot{\xi}^{\pm}$.
Since $\bm{I}_2=\mathrm{diag}\,(0,1,1,0)$, the normal derivatives for the noncharacteristic components $W^{\rm nc\,\pm}=(W_2^\pm,W_3^\pm)^{\mathsf{T}}$ are recovered directly from \eqref{W.eq} as:
\begin{equation}\label{dx2W2W3}
{\bm{I}}_2\partial_{x_2}W^{\pm}=F^\pm-{\bm{A}}_0^{r,l}\partial_t W^\pm-{\bm{A}}_1^{r,l}\partial_{x_1}W^\pm-{\bm{C}}^{r,l} W^\pm\,.
\end{equation}
The ``missing'' normal derivatives $\p_{x_2} W^{\pm}_1$ can be expressed by $\dot{\xi}^{\pm}$ and equations \eqref{W.eq}.
From transformation \eqref{W.def} and definition \eqref{T}, we have
\begin{align*}
&\partial_{x_2}\dot v^\pm=\partial_{x_2}W^\pm_1-\frac{c_{r,l}}{{\bm\gamma} p_{r,l}}\partial_{x_1}\Phi_{r,l}\left(\partial_{x_2}W^\pm_2-\partial_{x_2}W^\pm_3\right)+\big((\partial_{x_2}T_{r,l})W^\pm\big)_2 ,\\
&\partial_{x_2}\dot u^\pm=\partial_{x_1}\Phi_{r,l}\partial_{x_2}W^\pm_1+\frac{c_{r,l}}{{\bm\gamma} p_{r,l}}\left(\partial_{x_2}W^\pm_2-\partial_{x_2}W^\pm_3\right)+\big((\partial_{x_2}T_{r,l})W^\pm\big)_3 .
\end{align*}
In view of definition \eqref{xi.def}, we obtain
\begin{align}\label{dx2W1}
\langle\partial_{x_1}\Phi_{r,l}\rangle^2\partial_{x_2}W^\pm_1 = \partial_{x_2}\Phi_{r,l}(\partial_{x_1}\dot u^\pm-\dot\xi^\pm)-\partial_{x_1}\Phi_{r,l}\big((\partial_{x_2}T_{r,l})W^\pm\big)_3-\big((\partial_{x_2}T_{r,l})W^\pm\big)_2 .
\end{align}
Gathering \eqref{dx2W2W3} and \eqref{dx2W1} yields
\begin{align}\label{dx2W1-3}
\partial_{x_2}\begin{pmatrix} W^{\pm}_1\\ W^\pm_2\\ W^\pm_3\end{pmatrix}
=\begin{pmatrix} 0\\F_2^{\pm}\\ F_3^{\pm}\end{pmatrix}-\widetilde{{\bm{A}}}_0^{r,l}\partial_t W^\pm-\widetilde{{\bm{A}}}_1^{r,l}\partial_{x_1}W^\pm-\widetilde{{\bm{C}}}^{r,l} W^\pm-\frac{\partial_{x_2}\Phi_{r,l}}{\langle\partial_{x_1}\Phi_{r,l}\rangle^2}\begin{pmatrix}\dot\xi^\pm\\ 0\\ 0\end{pmatrix} ,
\end{align}
with new matrices $\widetilde{\bm{A}}_{0,1}^{r,l}$ and $\widetilde{\bm{C}}^{r,l}$, where
$\widetilde{\bm{A}}_{0,1}^{r,l}$ are $C^\infty$--functions of $(\dot{U}_{r,l},\nabla\dot{\Phi}_{r,l})$ and $\widetilde{\bm{C}}^{r,l}$ are $C^\infty$--functions of $(\dot{U}_{r,l},\nabla\dot{U}_{r,l},\nabla\dot{\Phi}_{r,l},\nabla^2\dot\Phi_{r,l})$ that
vanish at the origin.

From \eqref{dx2W1-3}, by an induction argument similar to Coulombel--Secchi \cite[Proposition 3]{CS08MR2423311}, we can get an estimate of the $L^2_\gamma(\Omega_T)$--norm of all normal derivatives up to order $s$ of $(W^\pm_1,W^\pm_2,W^\pm_3)$. More precisely, we have the following lemma.
\begin{lemma}\label{lem.tame3}
	Assume that the hypotheses of Theorem {\rm \ref{thm:L}} hold.
	Then there exist constants $C_{s}>0$ and $\gamma_{s}\geq 1$,
	which are independent of $T$, such that,
	for all $\gamma\geq \gamma_{s}$ and solutions $(W^\pm,\psi)\in H^{s+2}_{\gamma}(\Omega_T)\times H^{s+2}_{\gamma}(\omega_T)$
	of problem \eqref{W.eq} and \eqref{W.bdy},
	the following estimate holds for all integer $k\in[1,s]${\rm :}
	\begin{align}
	&\big\|\p_{x_2}^kW^{\pm}\big\|_{L^2(H^{s-k}_{\gamma}(\omega_T))}\leq  C_{s}\Big\{ \big\|\big(F^{\pm},W^{\pm},\dot{\xi}^{\pm}\big)\big\|_{H^{s-1}_{\gamma}(\Omega_T)}
	+\big\|W^{\pm}\big\|_{L^2(H^{s}_{\gamma}(\omega_T))}
	\notag \\
	&\qquad +\big\|\dot{\xi}^{\pm}\big\|_{L^{\infty}(\Omega_T)}\big\|\nabla\dot\Phi_{r,l} \big\|_{H^{s-1}_{\gamma}(\Omega_T)} +\big\|W^{\pm}\big\|_{L^{\infty}(\Omega_T)}\big\|\big(\dot U_{r,l},\nabla\dot\Phi_{r,l}\big) \big\|_{H^{s}_{\gamma}(\Omega_T)}\Big\}.
	\notag 
	\end{align}
\end{lemma}
It remains to obtain an estimate for the normal derivatives of the last components $W^\pm_4$ of $W^\pm$.
To this end, it is sufficient to notice from \eqref{T} that $W^\pm_4=\dot s^\pm$ (namely the last components of $W^\pm$ coincides with those of original unknowns $\dot V^\pm$). According to the equations for $\dot s^\pm$ in \eqref{ELP.1}--\eqref{ELP.1.1}, we have
\[
\partial_t W_4^\pm+v_{r,l}\partial_{x_1}W_4^\pm+(\mathcal{C}_{r,l}\dot V^\pm)_4=f^\pm_4 ,
\]
which is a transport equation. Arguing on it by the standard energy method leads to the following lemma.
\begin{lemma}[Estimate of entropy]\label{lem.tame4}
	Assume that the hypotheses of Theorem {\rm \ref{thm:L}} hold.
	Then there exist constants $C_{s}>0$ and $\gamma_{s}\geq 1$,
	which are independent of $T$, such that,
	for all $\gamma\geq \gamma_{s}$ and solutions $(W^\pm,\psi)\in H^{s+2}_{\gamma}(\Omega_T)\times H^{s+2}_{\gamma}(\omega_T)$
	of problem \eqref{W.eq} and \eqref{W.bdy},
	the following estimate holds:
	\begin{align}\notag 
	\gamma\Vert W_4^\pm\Vert_{H^s_\gamma(\Omega_T)}\leq C_s \left\{\Vert f^\pm_4\Vert_{H^s_\gamma(\Omega_T)}+\Vert\dot V\Vert_{H^s_\gamma(\Omega_T)}+\Vert (\dot U_{r,l},\nabla\dot\Phi_{r,l})\Vert_{H^{s+1}_\gamma(\Omega_T)}\Vert\dot V\Vert_{L^\infty(\Omega_T)}\right\}.
	\end{align}
\end{lemma}

In light of definition \eqref{norm.def}, we see that, for all $s\in\mathbb{N}$ and $\theta \in H^{s}_{\gamma}(\Omega_T)$,
\begin{align*}
\|\theta\|_{H^{s}_{\gamma}(\Omega_T)}^2=\sum_{k=0}^{s}\|\p_{x_2}^k\theta \|_{L^2(H^{s-k}_{\gamma}(\omega_T))}^2,
\qquad \gamma \|\theta\|_{H^{s-1}_{\gamma}(\Omega_T)}\leq \|\theta\|_{H^{s}_{\gamma}(\Omega_T)}.
\end{align*}
Thanks to these identities, we combine Lemmas \ref{lem.tame1}--\ref{lem.tame4}
and use the Moser-type inequalities to deduce the following \emph{a priori} estimates on the $H^{s}_{\gamma}$--norm of solution $\dot{V}^{\pm}$ to the linearized problem \eqref{ELP}.
\begin{proposition}\label{pro.tame}
	Assume that the hypotheses of Theorem {\rm \ref{thm:L}} hold.
	Then there exists a constant $K_0>0$ {\rm (}independent of $s$ and $T${\rm )}
	and constants $C_{s}>0$ and $\gamma_{s}\geq 1$ {\rm (}depending on $s$, but independent of $T${\rm )} such that,
	if $K\leq K_0$, then, for all $\gamma\geq \gamma_{s}$ and
	solutions $(\dot{V}^\pm,\psi)\in H^{s+2}_{\gamma}(\Omega_T)\times H^{s+2}_{\gamma}(\omega_T)$ of problem \eqref{ELP},
	the following estimate holds{\rm :}
	\begin{align}
	\notag &\sqrt{\gamma}\big\|\dot{V}\big\|_{H^{s}_{\gamma}(\Omega_T)}
	+\big\|\mathbb{P}({\varphi})\dot{V}|_{x_2=0}\big\|_{H^{s}_{\gamma}(\omega_T)}
	+\|\psi\|_{H^{s+1}_{\gamma}(\omega_T)}\\
	\notag &\leq C_{s}\Big\{\frac1{\sqrt\gamma}\left\|f\right\|_{H^{s}_{\gamma}(\Omega_T)}
	+\frac1{\gamma^{3/2}}\left\|f\right\|_{L^2(H^{s+1}_{\gamma}(\omega_T))}+ \frac1{\gamma}\|g\|_{H^{s+1}_{\gamma}(\omega_T)}\\
		&\qquad\quad +\frac1{\gamma}\big(\big\|\mathbb{P} ( {\varphi})\dot{V} |_{x_2=0}\big\|_{L^{\infty}(\omega_T)}
	+\|\psi\|_{W^{1,\infty}(\omega_T)}\big)\|(\dot U_{r,l},\partial_{x_2}\dot U_{r,l},\nabla\dot\Phi_{r,l})|_{x_2=0}\|_{{H^{s+1}_{\gamma}(\omega_T)}}
	\notag\\
	& \qquad\quad +\frac1{\gamma^{3/2}}\big( \left\|f\right\|_{L^{\infty}(\Omega_T)} +\big\|\dot{V}\big\|_{W^{1,\infty}(\Omega_T)}\big)\big\|\big(\dot U_{r,l},\nabla\dot\Phi_{r,l}\big)\big\|_{H^{s+2}_{\gamma}(\Omega_T)}\Big\}.
	\label{tame5}
	\end{align}
\end{proposition}

\subsection{Proof of Theorem \ref{thm:L}}
According to Theorem \ref{thm:L1}, the effective linear problem \eqref{ELP} is well-posed for sources
terms $(f^{\pm},g)\in L^2(H^1(\omega_T))\times H^1(\omega_T)$ that vanish in the past.
Following \cite{RM74MR0340832,CP82MR678605}, we can use Proposition \ref{pro.tame}
to convert Theorem \ref{thm:L1} into a well-posedness result for \eqref{ELP} in $H^{s}$.
More precisely, we can prove that, under the assumptions of Theorem \ref{thm:L},
if $(f^{\pm},g)\in H^{s+1}(\Omega_T)\times H^{s+1}(\omega_T)$ vanish in the past,
then there exists a unique solution $(\dot{V}^{\pm},\psi)\in H^{s}(\Omega_T)\times H^{s+1}(\omega_T)$
that vanishes in the past and satisfies \eqref{tame5} for all $\gamma\geq \gamma_{s}$.

It remains to prove the tame estimate \eqref{tame}.
To this end, we first fix the value of $\gamma$ such that $\gamma$ is greater than $\max\{\gamma_{3},\ldots,\gamma_{\tilde{\alpha}}\}$.
Using \eqref{tame5} with $s=3$ and \eqref{H.L.1}, we have
\begin{align}
\notag &\big\|\dot{V} \big\|_{H^3_{\gamma}(\Omega_T)}
+\big\|\mathbb{P}( {\varphi})\dot{V} |_{x_2=0}\big\|_{H^3_{\gamma}(\omega_T)}+\|\psi\|_{H^{4}_{\gamma}(\omega_T)}\notag\\
&\lesssim
K \big(\left\|f \right\|_{L^{\infty}(\Omega_T)}+\big\|\dot{V} \big\|_{W^{1,\infty}(\Omega_T)}
+\big\|\mathbb{P} ( {\varphi})\dot{V} |_{x_2=0}\big\|_{L^{\infty}(\omega_T)}+\|\psi\|_{W^{1,\infty}(\omega_T)}\big)\notag\\
&\quad +\left\|f \right\|_{H^{4}_{\gamma}(\Omega_T)}+\|g\|_{H^{4}_{\gamma}(\omega_T)}.
\label{tame.p1}
\end{align}
Note that $T$ and $\gamma$ have been fixed.
By virtue of classical Sobolev inequalities
$\|\theta\|_{L^{\infty}(\Omega_T)}\lesssim \|\theta\|_{H^2(\Omega_T)}$
and $\|\theta\|_{L^{\infty}(\omega_T)}\lesssim \|\theta\|_{H^2(\omega_T)}$,
we also get $\left\|f \right\|_{L^{\infty}(\Omega_T)}\lesssim  \left\|f \right\|_{H^{4}_{\gamma}(\Omega_T)}$, hence
\begin{equation}\label{sobolev.gamma}
\begin{split}
&\big\|\dot{V} \big\|_{W^{1,\infty}(\Omega_T)}
+\big\|\mathbb{P} ( {\varphi})\dot{V} |_{x_2=0}\big\|_{L^{\infty}(\omega_T)}+\|\psi\|_{W^{1,\infty}(\omega_T)}
\\ &\quad \lesssim  \big\|\dot{V} \big\|_{H^3_{\gamma}(\Omega_T)}
+\big\|\mathbb{P} ( {\varphi})\dot{V} |_{x_2=0}\big\|_{H^3_{\gamma}(\omega_T)}+\|\psi\|_{H^{4}_{\gamma}(\omega_T)}\,.
\end{split}
\end{equation}
We utilize \eqref{tame.p1} together with \eqref{sobolev.gamma} and take $K>0$ sufficiently small to obtain
\begin{align*}
&\big\|\dot{V} \big\|_{W^{1,\infty}(\Omega_T)}
+\big\|\mathbb{P} ( {\varphi})\dot{V} |_{x_2=0}\big\|_{L^{\infty}(\omega_T)}+\|\psi\|_{W^{1,\infty}(\omega_T)}\\
 &\quad\lesssim \left\|f \right\|_{H^{4}_{\gamma}(\Omega_T)}+\|g\|_{H^{4}_{\gamma}(\omega_T)}.
\end{align*}
Plugging these estimates into \eqref{tame5} yields \eqref{tame}, which completes the proof of Theorem \ref{thm:L}.

\section{Approximate solution}\label{sct:AS}
This section is devoted to the construction of the so-called approximate solution, which enables us to reduce the original problem \eqref{Phi.a}--\eqref{E0} and \eqref{Phi.b} into a nonlinear one with zero initial data.
The reformulated problem is expected to be solved in the space of functions vanishing in the past,
where we have already established a well-posedness result for the linearized problem, \emph{cf.}\;Theorem \ref{thm:L}.
The necessary compatibility conditions have to be prescribed on the initial data $(U_0^{\pm},\varphi_0)$
for the existence of smooth approximate solutions $(U^{a\,\pm},\Phi^{a\,\pm})$,
which are solutions of problem \eqref{Phi.a}--\eqref{E0} and \eqref{Phi.b} in the sense of Taylor's series at time $t=0$.

Let $s\geq 3$ be an integer.
Assume that $\dot{U}_0^{\pm}:=U_0^{\pm}-\widebar{U}^{\pm}\in H^{s+1/2}(\mathbb{R}_+^2)$
and  $\varphi_0\in H^{s+1}(\mathbb{R})$.
We also assume without loss of generality that $(\dot{U}_0^{\pm},\varphi_0)$ has the compact support:
\begin{align} \label{CA1}
\supp \dot{U}_0^{\pm}\subset \{x_2\geq 0,\, x_1^2+x_2^2\leq 1\},\qquad \supp \varphi_0\subset [-1,1].
\end{align}
We extend $\varphi_0$ from $\mathbb{R}$ to $\mathbb{R}_+^2$ by constructing
$\dot{\Phi}_0^\pm\in H^{s+3/2}(\mathbb{R}_+^2)$ such that
\begin{align} \label{CA2}
\dot{\Phi}_0^{\pm}\,\!|_{x_2=0}
=\varphi_0,\quad \supp \dot{\Phi}_0^{\pm}\subset \{x_2\geq 0,\, x_1^2+x_2^2\leq 2\},
\end{align}
and
\begin{align} \label{CA3}
\big\|\dot{\Phi}_0^{\pm}\big\|_{H^{s+3/2}(\mathbb{R}_+^2)}
\leq C\|\varphi_0\|_{H^{s+1}(\mathbb{R})}.
\end{align}
For problem \eqref{Phi.a} and \eqref{Phi.b}, we prescribe the initial data:
\begin{align} \label{Phi.initial}
\Phi^{\pm}|_{t=0}=\Phi_0^{\pm}:=\dot{\Phi}_0^{\pm}+\widebar{\Phi}^{\pm}\,,
\end{align}
where $\widebar{\Phi}^{\pm}$ are defined in \eqref{CVS0}.
The estimate \eqref{CA3} and  the Sobolev embedding theorem yield
\begin{align} \label{CA4}
\pm \p_{x_2}\Phi_0^{\pm}\geq {7}/{8}\qquad\,\, \textrm{for all }x\in\mathbb{R}_+^2,
\end{align}
provided $\varphi_0$ is sufficiently small in $H^{s+1}(\mathbb{R})$.

Let us denote the perturbation by
$(\dot{U}^{\pm},\dot{\Phi}^{\pm}):=(U^{\pm}-\widebar{U}^{\pm},\Phi^{\pm}-\widebar{\Phi}^{\pm})$,
and  the traces of the ${k}$-th order time derivatives on $\{t=0\}$  by
$$
\dot{U}^{\pm}_{{k}}:=\p_t^{{k}}\dot{U}^{\pm}|_{t=0},\quad
\dot{\Phi}^{\pm}_{{k}}:=\p_t^{{k}}\dot{\Phi}^{\pm}|_{t=0},\qquad\,\, {k}\in\mathbb{N}.
$$

To introduce the compatibility conditions, we need to determine traces $\dot{U}^{\pm}_{{k}}$ and $\dot{\Phi}^{\pm}_{{k}}$
in terms of the initial data $\dot{U}^{\pm}_0$ and $\dot{\Phi}^{\pm}_0$ through equations \eqref{E0.a}--\eqref{E0.a1} and \eqref{Phi.b}.
For this purpose, we set $\mathcal{W}^{\pm}:=(\dot{U}^{\pm},\nabla_x\dot{U}^{\pm},\nabla_x\dot{\Phi}^{\pm})\in\mathbb{R}^{14}$,
and rewrite  \eqref{E0.a}--\eqref{E0.a1}, \eqref{Phi.b} (for $(U^\pm,\Phi^\pm)=(\widebar{U}^\pm+\dot U^\pm,\widebar{\Phi}^\pm+\dot\Phi^\pm)$) as
\begin{align}\label{tilde.U.Phi}
\p_t \dot{U}^{\pm}=\mathbf{F}_1(\mathcal{W}^{\pm}),\qquad \p_t \dot{\Phi}^{\pm}=\mathbf{F}_2(\mathcal{W}^{\pm}),
\end{align}
where $\mathbf{F}_1$ and $\mathbf{F}_2$ are suitable $C^{\infty}$--functions that vanish at the origin.
Applying operator $\p^{k}_t$ to \eqref{tilde.U.Phi} and taking the traces at time $t=0$,
one can employ the generalized Fa\`a di Bruno's formula ({\it cf}. \cite[Theorem 2.1]{M00MR1781515})
to derive
\begin{align} \label{tilde.U.0}
&\dot{U}^{\pm}_{{k}+1}
=\sum_{\alpha_{i}\in\mathbb{N}^{11},|\alpha_1|+\cdots+{k} |\alpha_{{k}}|={k}}
D^{\alpha_1+\cdots+\alpha_{k}}\mathbf{F}_1(\mathcal{W}^{\pm}_{0})\prod_{i=1}^{k}\frac{{k}!}{\alpha_{i}!}
\left(\frac{\mathcal{W}_{i}^{\pm}}{i!}\right)^{\alpha_{i}},\\ \label{tilde.Phi.0}
&\dot{\Phi}^{\pm}_{{k}+1}
=\sum_{\alpha_{i}\in\mathbb{N}^{11},|\alpha_1|+\cdots+{k} |\alpha_{{k}}|={k}}
D^{\alpha_1+\cdots+\alpha_{k}}\mathbf{F}_2(\mathcal{W}^{\pm}_{0})\prod_{i=1}^{k}\frac{{k}!}{\alpha_{i}!}
\left(\frac{\mathcal{W}_{i}^{\pm}}{i!}\right)^{\alpha_{i}},
\end{align}
where $\mathcal{W}_{i}^{\pm}$ denotes trace $(\dot{U}_{i}^{\pm},\nabla_x\dot{U}_{i}^{\pm},\nabla_x\dot{\Phi}_{i}^{\pm})$
at $t=0$.
From \eqref{tilde.U.0}--\eqref{tilde.Phi.0}, one can determine
$(\dot{U}^{\pm}_{{k}},\dot{\Phi}^{\pm}_{{k}})_{{k}\geq 0}$ inductively as functions of
the initial data $\dot{U}^{\pm}_{0}$ and $\dot{\Phi}^{\pm}_{0}$.
Furthermore, we have the following lemma (see \cite[Lemma 4.2.1]{M01MR1842775} for the proof):

\begin{lemma} \label{lem.Metivier}
	Let $s\ge 3$ be an integer. Assume that \eqref{CA1}--\eqref{CA3} 　and \eqref{CA4} hold.
	Then equations \eqref{tilde.U.0}--\eqref{tilde.Phi.0} determine
	$\dot{U}^{\pm}_{{k}}\in H^{s+1/2-{k}}(\mathbb{R}_+^2)$ for ${k}=1,\ldots,s$,
	and $\dot{\Phi}^{\pm}_{{k}}\in H^{s+3/2-{k}}(\mathbb{R}_+^2)$ for ${k}=1,\ldots,s+1$,
	such that
	\begin{align*}
	\supp \dot{U}_{{k}}^{\pm}\subset \{x_2\geq 0,\, x_1^2+x_2^2\leq 1\}, \qquad
	\supp \dot{\Phi}_{{k}}^{\pm}\subset \{x_2\geq 0,\, x_1^2+x_2^2\leq 2\}.
	\end{align*}
	Moreover,
	\begin{align} \notag
	\sum_{{k}=0}^{s}\big\|\dot{U}^{\pm}_{{k}}\big\|_{H^{s+1/2-{k}}(\mathbb{R}_+^2)}
	+\sum_{{k}=0}^{s+1}\big\|\dot{\Phi}^{\pm}_{{k}}\big\|_{H^{s+3/2-{k}}(\mathbb{R}_+^2)}
	\leq C\big(\big\|\dot{U}^{\pm}_0\big\|_{H^{s+1/2}(\mathbb{R}_+^2)}+\|\varphi_0\|_{H^{s+1}(\mathbb{R})} \big),
	\end{align}
	where constant $C>0$ depends only on $s$
	and $\|(\dot{U}^{\pm}_{0},\dot{\Phi}^{\pm}_{0})\|_{W^{1,\infty}(\mathbb{R}_+^2)}$.
\end{lemma}

To construct a \emph{smooth} approximate solution, one has to impose certain assumptions on
traces $\dot{U}^{\pm}_{{k}}$ and $\dot{\Phi}^{\pm}_{{k}}$.
We are now ready to introduce the notion of the {\it compatibility condition}.

\begin{definition}\label{def:CC}
	Let $s\geq 3$ be an integer.
	Let $\dot{U}^{\pm}_0:=U_0^{\pm}-\widebar{U}_0^{\pm}\in H^{s+1/2}(\mathbb{R}_+^2)$ and $\varphi_0\in H^{s+1}(\mathbb{R})$
	satisfy \eqref{CA1}.
	The initial data $U_0^{\pm}$ and $\varphi_0$ are said to be compatible up to order $s$
	if there exist functions $\dot{\Phi}_0^{\pm}\in H^{s+3/2}(\mathbb{R}_+^2)$
	satisfying \eqref{CA2}--\eqref{CA4}
	such that
	functions $\dot{U}^{\pm}_1,\ldots,\dot{U}^{\pm}_{s},\dot{\Phi}^{\pm}_1,\ldots,\dot{\Phi}^{\pm}_{s+1}$
	determined by \eqref{tilde.U.0}--\eqref{tilde.Phi.0} satisfy:
	\begin{subequations} \label{compa1}
		\begin{alignat}{2}
		&\p_{x_2}^{j}\big(\dot{\Phi}^{+}_{{k}}-\dot{\Phi}^{-}_{{k}}\big)|_{x_2=0}=0
		\qquad\,\, && \textrm{for }j,{k}\in\mathbb{N}\textrm{ with } j+{k}< s+1,\\
		&\p_{x_2}^{j}\big(\dot{p}^{+}_{{k}}-\dot{p}^{-}_{{k}}\big)|_{x_2=0}=0
		\qquad\,\, && \textrm{for }j,{k}\in\mathbb{N} \textrm{ with } j+{k}< s,
		\end{alignat}
	\end{subequations}
	and
	\begin{subequations}\label{compa2}
		\begin{alignat}{2}
		&\int_{\mathbb{R}_+^2}\big|\p_{x_2}^{s+1-{k}}\big(\dot{\Phi}^{+}_{{k}}-\dot{\Phi}^{-}_{{k}}\big)\big|^2\d x_1\frac{\d x_2}{x_2}
		<\infty\qquad\,\, && \textrm{for }{k}=0,\ldots,  s+1,\\
		&\int_{\mathbb{R}_+^2}\big|\p_{x_2}^{s-{k}}\big(\dot{p}^{+}_{{k}}-\dot{p}^{-}_{{k}}\big)\big|^2\d x_1\frac{\d x_2}{x_2}
		<\infty\qquad\,\, && \textrm{for }{k}=0,\ldots, s.
		\end{alignat}
	\end{subequations}
\end{definition}

It follows from Lemma \ref{lem.Metivier}
that $\dot{p}^{\pm}_{0},\ldots,\dot{p}^{\pm}_{s-2},\dot{\Phi}^{\pm}_{0},\ldots,
\dot{\Phi}^{\pm}_{s-1}\in  H^{5/2}(\mathbb{R}_+^2)\subset W^{1,\infty}(\mathbb{R}_+^2)$.
Then we can take the $j$-th order derivatives of the traces in \eqref{compa1}.
Relations \eqref{compa1} and \eqref{compa2}  enable us to utilize the lifting result
in \cite[Theorem 2.3]{LM72MR0350178} to construct the approximate solution in the following lemma.
We refer to \cite[Lemma 3]{CS08MR2423311} for the  proof.

\begin{lemma} \label{lem.app}
	Let $s\geq 3$ be an integer.
	Assume that $\dot{U}^{\pm}_0:=U_0^{\pm}-\widebar{U}_0^{\pm}\in H^{s+1/2}(\mathbb{R}_+^2)$
	and $\varphi_0\in H^{s+1}(\mathbb{R})$ satisfy \eqref{CA1},
	and that the initial data $U_0^{\pm}$ and $\varphi_0$ are compatible up to order $s$.
	If $\dot{U}^{\pm}_0$ and $\varphi_0$ are sufficiently small,
	then there exist functions $U^{a\pm}$, $\Phi^{a\pm}$, and $\varphi^a$ such that
	$\dot{U}^{a\pm}:=U^{a\pm}-\widebar{U}^{\pm}\in H^{s+1}(\Omega)$,
	$\dot{\Phi}^{a\pm}:=\Phi^{a\pm}-\widebar{\Phi}^{\pm}\in H^{s+2}(\Omega)$,
	$\varphi^a\in H^{s+3/2}(\p\Omega)$, and
	\begin{subequations} \label{app}
		\begin{alignat}{2}
		\label{app.1}&\p_t\Phi^{a\pm}+v^{a\pm}\p_{x_1}\Phi^{a\pm}-u^{a\pm}=0\qquad &&\textrm{in }\Omega,\\
		\label{app.2}&\p_t^j\mathbb{L}(U^{a\pm},\Phi^{a\pm})|_{t=0}=0\qquad &&\textrm{for }j=0,\ldots,s-1,\\
		\label{app.3}&\Phi^{a+}=\Phi^{a-}=\varphi^a\qquad &&\textrm{on }\p\Omega,\\
		\label{app.4}&\mathbb{B}(U^{a+},U^{a-},\varphi^a)=0\qquad &&\textrm{on }\p\Omega.
		\end{alignat}
	\end{subequations}
	Furthermore, we have
	\begin{align}
	\label{app2}&\pm \p_{x_2}\Phi^{a\pm}\geq {3}/{4}\quad \textrm{for all }(t,x)\in\Omega,\\
	\label{app4}&\supp \big(\dot{U}^{a\pm},\dot{\Phi}^{a\pm} \big)\subset \left\{t\in[-T,T],\,x_2\geq 0,\,x_1^2+x_2^2\leq 3 \right\},
	\end{align}
	and
	\begin{align}
	\big\|\dot{U}^{a\pm}\big\|_{H^{s+1}(\Omega)}
	+\big\|\dot{\Phi}^{a\pm}\big\|_{H^{s+2}(\Omega)}+\|\varphi^a\|_{H^{s+3/2}(\p\Omega)}
	\leq\varepsilon_0\big(\big\|\dot{U}^{\pm}_0\big\|_{H^{s+1/2}(\mathbb{R}_+^2)}
	+\|\varphi_0\|_{H^{s+1}(\mathbb{R})}\big), \label{app3}
	\end{align}
where we denote by $\varepsilon_0(\cdot)$  a function that tends to $0$ when its argument tends to $0$.
\end{lemma}

Vectors $(U^{a\,\pm},\Phi^{a\,\pm})$ in Lemma \ref{lem.app} are called the {\it approximate solution} to problem \eqref{Phi.a}--\eqref{E0} and \eqref{Phi.b}.
It follows from relations \eqref{app.3} and \eqref{app4} that $\varphi^a$ is supported within $\{-T\le t\le T,\,x_1^2\leq 3\}$.
Since $s\geq 3$, it follows from \eqref{app3} and the Sobolev embedding theorem that
\begin{align} \notag
\big\|\dot{U}^{a\pm}\big\|_{W^{2,\infty}(\Omega)}+\big\|\dot{\Phi}^{a\pm}\big\|_{W^{3,\infty}(\Omega)}
\leq\varepsilon_0\big(\big\|\dot{U}^{\pm}_0\big\|_{H^{s+1/2}(\mathbb{R}_+^2)}+\|\varphi_0\|_{H^{s+1}(\mathbb{R})}\big).
\end{align}

Let us reformulate the original problem into that with zero initial data by utilizing the approximate solution $(U^{a\,\pm},\Phi^{a\,\pm})$.
For this purpose, we introduce
\begin{align}\label{f.a}
f^{a\,\pm}:=\left\{\begin{aligned}
& -\mathbb{L}(U^{a\,\pm},\Phi^{a\,\pm}) \qquad &\textrm{if }t>0,\\
& 0 \qquad &\textrm{if }t<0.
\end{aligned}\right.
\end{align}
Since $(\dot{U}^{a\pm},\nabla\dot{\Phi}^{a\pm})\in  H^{s+1}(\Omega)$,
taking into account \eqref{app.2} and \eqref{app4}, we obtain that $f^{a\,\pm}\in H^{s}(\Omega)$ and
\begin{align} \notag
\supp f^{a\,\pm}\subset \left\{0\le t\le T,\,x_2\geq 0,\,x_1^2+x_2^2\leq 3 \right\}.
\end{align}
Using the Moser-type inequalities and the fact that $f^{a\,\pm}$ vanish in the past,
we obtain from \eqref{app3} the estimate:
\begin{align}\label{app5}
\|f^{a\,\pm}\|_{ H^{s}(\Omega)}\leq \varepsilon_0\big(\big\|\dot{U}^{\pm}_0\big\|_{H^{s+1/2}(\mathbb{R}_+^2)}
+\|\varphi_0\|_{H^{s+1}(\mathbb{R})}\big).
\end{align}

Let $(U^{a\,\pm},\Phi^{a\,\pm})$ be the approximate solution defined in Lemma \ref{lem.app}.
By virtue of \eqref{app} and \eqref{f.a}, we see that $(U^\pm,\Phi^\pm)=(U^{a\,\pm},\Phi^{a\,\pm})+(V^\pm,\Psi^\pm)$ is a solution of the original problem \eqref{Phi.a}--\eqref{E0} and \eqref{Phi.b} on $[0,T]\times \mathbb{R}_+^2$,
if $(V^\pm, \Psi^\pm)$ solve the following problem:
\begin{align} \label{P.new}
\left\{\begin{aligned}
&\mathcal{L}(V^\pm,\Psi^\pm):=\mathbb{L}(U^{a\,\pm}+V^\pm,\Phi^{a\,\pm}+\Psi^\pm)-\mathbb{L}(U^{a\,\pm},\Phi^{a\,\pm})=f^{a\,\pm}\quad&&\textrm{in }\Omega_T,\\
&\mathcal{E}(V^\pm,\Psi^\pm):=\p_t\Psi^\pm+(v^{a\,\pm}+v^\pm)\p_{x_1}\Psi^\pm+v^\pm\p_{x_1}\Phi^{a\,\pm}-u^\pm=0\quad&&\textrm{in }\Omega_T,\\
&\mathcal{B}(V,\psi):=\mathbb{B}(U^{a\,+}+V^+, U^{a\,-}+V^-, \varphi^a+\psi)=0\quad&&\textrm{on }\omega_T,\\
&\Psi^+=\Psi^-=\psi\quad&&\textrm{on }\omega_T,\\
&(V^\pm,\Psi^\pm)=0,\quad&&\textrm{for }t< 0.
\end{aligned}\right.
\end{align}
The initial data \eqref{E0.c} and \eqref{Phi.initial} have been absorbed into the interior equations.
From \eqref{app.1} and \eqref{app.4}, we observe that $(V^\pm,\Psi^\pm)=0$ satisfies \eqref{P.new} for $t<0$.
Therefore, the original nonlinear problem on $[0,T]\times \mathbb{R}_+^2$ is now reformulated
as a problem on $\Omega_T$ whose solutions vanish in the past.

\section{Nash--Moser Iteration}\label{sct:N-M}
In this section, we are going to solve the nonlinear problem, stated in the equivalent form \eqref{P.new}, by a suitable iteration scheme of Nash--Moser type (\emph{cf.}\;\citet{H76MR0602181} and the references therein).

In order to keep the notation manageable, from now on, it is undestood that functions defined in the interior domain have $+$ and $-$ states. For instance, unless otherwise explicitly written, we write simply $(U,\Phi)$, $(V,\Psi)$, by meaning $(U^\pm,\Phi^\pm)$, $(V^\pm,\Psi^\pm)$. According to this shortcut notation, we shall write $\mathbb L(U,\Psi)$ for the pair $\mathbb L(U^\pm,\Psi^\pm)$.

\subsection{Iterative scheme}
Before describing the iterative scheme for problem \eqref{P.new}, we recall the following result for smoothing operators $S_{\theta}$ from \cite[Proposition 4]{CS08MR2423311}.
\begin{proposition}\label{pro.smooth}
	Let $T>0$ and $\gamma\geq 1$, and let $m\geq 4$ be an integer.
	Then there exists a family $\{S_{\theta}\}_{\theta\geq 1}$ of smoothing operators{\rm :}
	\begin{align*}
	S_{\theta}:\ \mathcal{F}_{\gamma}^3(\Omega_T)\times\mathcal{F}_{\gamma}^3(\Omega_T)
	\longrightarrow \bigcap_{\beta\geq 3}\mathcal{F}_{\gamma}^\beta(\Omega_T)\times\mathcal{F}_{\gamma}^\beta(\Omega_T),
	\end{align*}
	where $\mathcal{F}_{\gamma}^\beta(\Omega_T):=\big\{u\in H^{\beta}_{\gamma}(\Omega_T):u=0\textrm{ for }t<0\big\}$
	is a closed subspace of $H^{\beta}_{\gamma}(\Omega_T)$ such that
	\begin{subequations}\label{smooth.p1}
		\begin{alignat}{2}
		\label{smooth.p1a}&\|S_{\theta} u\|_{H^{\beta}_{\gamma}(\Omega_T)}\leq C\theta^{(\beta-\alpha)_+}\|u\|_{H^{\alpha}_{\gamma}(\Omega_T)}
		&&\quad\textrm{for all }\alpha,\beta\in[1,m],\\[1.5mm]
		\label{smooth.p1b}&\|S_{\theta} u-u\|_{H^{\beta}_{\gamma}(\Omega_T)}\leq C\theta^{\beta-\alpha}\|u\|_{H^{\alpha}_{\gamma}(\Omega_T)}
		&&\quad\textrm{for all }1\leq \beta\leq \alpha\leq m,\\
		\label{smooth.p1c}&\left\|\frac{\d}{\d \theta}S_{\theta} u\right\|_{H^{\beta}_{\gamma}(\Omega_T)}
		\leq C\theta^{\beta-\alpha-1}\|u\|_{H^{\alpha}_{\gamma}(\Omega_T)}&&\quad\textrm{for all }\alpha,\beta\in[1,m],
		\end{alignat}
	\end{subequations}
	and
	\begin{align}
	\label{smooth.p2}\|(S_{\theta}u-S_{\theta}v)|_{x_2=0}\|_{H^{\beta}_{\gamma}(\omega_T)}
	\leq C\theta^{(\beta+1-\alpha)_+}\|(u-v)|_{x_2=0}\|_{H^{\alpha}_{\gamma}(\omega_T)} \quad \, \textrm{for all }\alpha,\beta\in[1,m],
	\end{align}
	where $\alpha,\beta\in\mathbb{N}$, $(\beta-\alpha)_+:=\max\{0,\beta-\alpha\}$, and $C>0$ is a constant depending only on $m$.
	In particular, if $u=v$ on $\omega_T$, then $S_{\theta}u=S_{\theta}v$ on $\omega_T$.
	Furthermore, there exists another family of smoothing operators {\rm (}still denoted by $S_{\theta}${\rm )} acting on the functions
	defined on the boundary $\omega_T$ and satisfying the properties in \eqref{smooth.p1} with
	norms $\|\cdot\|_{H^{\alpha}_{\gamma}(\omega_T)}$.
\end{proposition}

The proof of \eqref{smooth.p2} is based on the following lifting
operator, which will be also useful in the sequel of our analysis (see \cite[Chapter 5]{FM00MR1787068} and \cite{CS08MR2423311} for the proof).
\begin{lemma} \label{lem.smooth2}
	Let $T>0$ and $\gamma\geq 1$, and let $m\geq 1$ be an integer.
	Then there exists an operator $\mathcal{R}_T$, which is continuous from $\mathcal{F}_{\gamma}^s(\omega_T)$
	to $\mathcal{F}_{\gamma}^{s+1/2}(\Omega_T)$ for all $s\in [1,m]$,
	such that, if $s\geq 1$ and $u\in  \mathcal{F}_{\gamma}^s(\omega_T)$,
	then $(\mathcal{R}_T u)|_{x_2=0}=u$.
\end{lemma}

We are going to follow \cite{CS08MR2423311} and describe the iteration scheme for problem \eqref{P.new}.

\vspace*{2mm}
\noindent{\bf Assumption\;(A-1)}: \emph{$ (V_0,\Psi_0, \psi_0)=0$ and, for $k=0,\ldots,{n}$,
	$(V_k,\Psi_k,\psi_k)$ are already given  and verify}
\begin{align}\label{NM.H1}
(V_k,\Psi_k,\psi_k)|_{t<0}=0,\quad \Psi_k^{+}|_{x_2=0}=\Psi_k^{-}|_{x_2=0}=\psi_k.
\end{align}

Let us consider
\begin{align}\label{NM0}
V_{{n}+1}=V_{{n}}+\delta V_{{n}},\quad \Psi_{{n}+1}=\Psi_{{n}}+\delta \Psi_{{n}},
\quad\,\, \psi_{{n}+1}=\psi_{{n}}+\delta \psi_{{n}},
\end{align}
where these differences $\delta V_{{n}}$, $\delta \Psi_{{n}}$, and $\delta \psi_{{n}}$ will be specified below.

First we find $(\delta \dot{V}_{{n}},\delta\psi_{{n}})$ by solving the effective linear problem:
\begin{align} \label{effective.NM}
\left\{\begin{aligned}
&\mathbb{L}_e'(U^a+V_{{n}+1/2},\Phi^a+\Psi_{{n}+1/2})\delta \dot{V}_{{n}}=f_{{n}}
\qquad &&\textrm{in }\Omega_T,\\
& \mathbb{B}_e'(U^a+V_{{n}+1/2},\Phi^a+\Psi_{{n}+1/2})(\delta \dot{V}_{{n}},\delta\psi_{{n}})=g_{{n}}
\qquad &&\textrm{on }\omega_T,\\
& (\delta \dot{V}_{{n}},\delta\psi_{{n}})=0\qquad &&\textrm{for }t<0,
\end{aligned}\right.
\end{align}
where operators  $\mathbb{L}_e'$, $\mathbb{B}_e'$ are defined in \eqref{ELP.1}--\eqref{ELP.1.1}, \eqref{ELP.2}, respectively,
\begin{align} \label{good.NM}
\delta \dot{V}_{{n}}:=\delta V_{{n}}-\frac{\p_{x_2} (U^a+V_{{n}+1/2})}{\p_{x_2} (\Phi^a+\Psi_{{n}+1/2})}\delta\Psi_{{n}}
\end{align}
is the ``good unknown'' (\emph{cf.}\;\eqref{Alinhac}), and $(V_{{n}+1/2},\Psi_{{n}+1/2})$ is a smooth {\em modified state}
such that $(U^a+V_{{n}+1/2},\Phi^a+\Psi_{{n}+1/2})$ satisfies
constraints \eqref{bas.c1}--\eqref{bas.eq}.
The source term $(f_{{n}},g_{{n}})$ will be defined through the accumulated errors
at Step ${n}$ later on.

Define the modified state with $V_{{n}+1/2}^{\pm}=(p_{{n}+1/2}^{\pm},v_{{n}+1/2}^{\pm},u_{{n}+1/2}^{\pm},s_{{n}+1/2}^{\pm})^{\mathsf{T}}$ as
\begin{align} \label{modified}
\left\{\begin{aligned}
&\Psi_{{n}+1/2}^{\pm}:=S_{\theta_{{n}}}\Psi_{{n}}^{\pm},\quad v_{{n}+1/2}^{\pm}
:=S_{\theta_{{n}}}v_n^{\pm},\quad s_{{n}+1/2}^{\pm}
:=S_{\theta_{{n}}}s_n^{\pm},\\
&p_{{n}+1/2}^{\pm}:=S_{\theta_{{n}}}p_{{n}}^{\pm}\mp \tfrac{1}{2}\mathcal{R}_T\big(S_{\theta_{{n}}}p_{{n}}^{+}|_{x_2=0}
-S_{\theta_{{n}}}p_{{n}}^{-}|_{x_2=0}  \big),\\
&u_{{n}+1/2}^{\pm}:=\p_t\Psi_{{n}+1/2}^{\pm}+\big(v^{a\pm}+v_{{n}+1/2}^{\pm}  \big)\p_{x_1}\Psi_{{n}+1/2}^{\pm}
+v_{{n}+1/2}^{\pm}\p_{x_1}\Phi^{a\pm},
\end{aligned}\right.
\end{align}
where $S_{\theta_{{n}}}$ are the smoothing operators defined in Proposition \ref{pro.smooth} with
sequence $\{\theta_{{n}}\}$ given by
\begin{align} \label{theta}
\theta_0\geq 1,\qquad \theta_{{n}}=\sqrt{\theta^2_0+{n}},
\end{align}
and $\mathcal{R}_T$ is the lifting operator given in Lemma \ref{lem.smooth2}.
In light of Proposition \ref{pro.smooth} and \eqref{NM.H1}, we obtain
\begin{align} \label{modified.1}
\left\{\begin{aligned}
&\Psi_{{n}+1/2}^{+}|_{x_2=0}=\Psi_{{n}+1/2}^{-}|_{x_2=0}=:\psi_{{n}+1/2},&&\\
&p_{{n}+1/2}^{+}|_{x_2=0}=p_{{n}+1/2}^{-}|_{x_2=0},&&\quad\mbox{on}\,\,\,\omega_T\,,\\
&\mathcal{E}(V_{{n}+1/2},\Psi_{{n}+1/2})=0,&&\quad\mbox{in}\,\,\,\Omega_T\,,\\
&\big(V_{{n}+1/2},\Psi_{{n}+1/2},\psi_{{n}+1/2} \big)|_{t<0}=0.&&
\end{aligned}\right.
\end{align}
It then follows from \eqref{app} that $(U^a+V_{{n}+1/2},\Phi^a+\Psi_{{n}+1/2})$
satisfies \eqref{bas.eq.1} and \eqref{bas.eq.3}--\eqref{bas.eq.4}.
The additional constraint \eqref{bas.eq.2} will be obtained by taking the initial data small enough.

The errors at Step ${n}$ are defined from the following decompositions:
\begin{align}
\notag&\mathcal{L}(V_{{n}+1},\Psi_{{n}+1})-\mathcal{L}(V_{{n}},\Psi_{{n}})\\
\notag&\quad = \mathbb{L}'(U^a+V_{{n}},\Phi^a+\Psi_{{n}})(\delta V_{{n}},\delta\Psi_{{n}})+e_{{n}}'\\
\notag&\quad = \mathbb{L}'(U^a+S_{\theta_{{n}}}V_{{n}},\Phi^a+S_{\theta_{{n}}}\Psi_{{n}})(\delta V_{{n}},\delta\Psi_{{n}})+e_{{n}}'+e_{{n}}''\\
\notag&\quad = \mathbb{L}'(U^a+V_{{n}+1/2},\Phi^a+\Psi_{{n}+1/2})(\delta V_{{n}},\delta\Psi_{{n}})+e_{{n}}'+e_{{n}}''+e_{{n}}'''\\
\label{decom1}&\quad = \mathbb{L}_e'(U^a+V_{{n}+1/2},\Phi^a+\Psi_{{n}+1/2})\delta \dot{V}_{{n}}+e_{{n}}'+e_{{n}}''+e_{{n}}'''+D_{{n}+1/2} \delta\Psi_{{n}}
\end{align}
and
\begin{align}
\notag&\mathcal{B}(V_{{n}+1}|_{x_2=0},\psi_{{n}+1})-\mathcal{B}(V_{{n}}|_{x_2=0},\psi_{{n}})\\
\notag&\quad = \mathbb{B}'(U^a+V_{{n}},\Phi^a+\Psi_{{n}})(\delta V_{{n}}|_{x_2=0},\delta\psi_{{n}})+\tilde{e}_{{n}}'\\
\notag&\quad = \mathbb{B}'(U^a+S_{\theta_{{n}}}V_{{n}},\Phi^a+S_{\theta_{{n}}}\Psi_{{n}})(\delta V_{{n}}|_{x_2=0},\delta\psi_{{n}})+\tilde{e}_{{n}}'+\tilde{e}_{{n}}''\\
\label{decom2}&\quad =\mathbb{B}_e'(U^a+V_{{n}+1/2},\Phi^a+\Psi_{{n}+1/2})(\delta \dot{V}_{{n}}|_{x_2=0},\delta\psi_{{n}})+\tilde{e}_{{n}}'+\tilde{e}_{{n}}''+\tilde{e}_{{n}}''',
\end{align}
where we have set
\begin{align}\label{error.D}
D_{{n}+1/2}:=\frac{1}{\p_2(\Phi^a+\Psi_{{n}+1/2})}\p_2\mathbb{L}(U^a+V_{{n}+1/2},\Phi^a+\Psi_{{n}+1/2}),
\end{align}
and have used \eqref{Alinhac} to obtain the last identity in \eqref{decom1}.
Let us set
\begin{align} \label{e.e.tilde}
e_{{n}}:=e_{{n}}'+e_{{n}}''+e_{{n}}'''+D_{{n}+1/2} \delta\Psi_{{n}},\qquad
\tilde{e}_{{n}}:=\tilde{e}_{{n}}'+\tilde{e}_{{n}}''+\tilde{e}_{{n}}'''.
\end{align}

\smallskip
\noindent{\bf Assumption\;(A-2)}: \emph{$f_0:=S_{\theta_0}f^a$, $(e_0,\tilde{e}_0,g_0):=0$, and $(f_k,g_k,e_k,\tilde{e}_k)$
	are already given and vanish in the past for $k=0,\ldots,{n}-1$.}

\vspace*{2mm}
We compute the accumulated errors at Step ${n}$,  $n\geq 1$, by
\begin{align}  \label{E.E.tilde}
E_{{n}}:=\sum_{k=0}^{{n}-1}e_{k},\quad \widetilde{E}_{{n}}:=\sum_{k=0}^{{n}-1}\tilde{e}_{k}.
\end{align}
Then we compute $f_{{n}}$ and $g_{{n}}$ for ${n}\geq 1$  from the equations:
\begin{align} \label{source}
\sum_{k=0}^{{n}} f_k+S_{\theta_{{n}}}E_{{n}}=S_{\theta_{{n}}}f^a,
\qquad \sum_{k=0}^{{n}}g_k+S_{\theta_{{n}}}\widetilde{E}_{{n}}=0.
\end{align}

Under assumptions {\bf (A-1)}--{\bf (A-2)},
$(V_{{n}+1/2},\Psi_{{n}+1/2})$ and $(f_{{n}},g_{{n}})$ have been specified from \eqref{modified} and \eqref{source}.
Then we can obtain $(\delta \dot{V}_{{n}},\delta\psi_{{n}})$ as the solution of
the linear problem  \eqref{effective.NM} by applying Theorem \ref{thm:L}.

Now we need to construct $\delta\Psi_{{n}}=(\delta\Psi_{{n}}^+,\delta\Psi_{{n}}^-)^{\mathsf{T}}$
satisfying $\delta\Psi_{{n}}^{\pm}|_{x_2=0}=\delta\psi_{{n}}$.
From the explicit expression of the boundary conditions in \eqref{effective.NM} (\emph{cf.}\;\eqref{b.bar}--\eqref{M.bar} and \eqref{good.NM}), we observe that $\delta\psi_n$ solves
\begin{align*}
&\partial_t\delta\psi_n+(v^{a\,+}+v^+_{n+1/2})\vert_{x_2=0}\partial_{x_1}\delta\psi_n\\
&\qquad +\left\{\partial_{x_1}(\varphi^a+\psi_{n+1/2})\frac{\partial_{x_2}(v^{a\,+}+v^+_{n+1/2})\vert_{x_2=0}}{\partial_{x_2}(\Phi^{a\,+}+\Psi^+_{n+1/2})\vert_{x_2=0}}-\frac{\partial_{x_2}(u^{a\,+}+u^+_{n+1/2})\vert_{x_2=0}}{\partial_{x_2}(\Phi^{a\,+}+\Psi^+_{n+1/2})\vert_{x_2=0}}\right\}\delta\psi_n\\
&\qquad +\partial_{x_1}(\varphi^a+\psi_{n+1/2})(\delta\dot v^+_n)\vert_{x_2=0}-(\delta\dot u^+_{n})\vert_{x_2=0}=g_{n,2} ,
\end{align*}
\begin{align*}
&\partial_t\delta\psi_n+(v^{a\,-}+v^-_{n+1/2})\vert_{x_2=0}\partial_{x_1}\delta\psi_n\\
&\qquad +\left\{\partial_{x_1}(\varphi^a+\psi_{n+1/2})\frac{\partial_{x_2}(v^{a\,-}+v^-_{n+1/2})\vert_{x_2=0}}{\partial_{x_2}(\Phi^{a\,-}+\Psi^-_{n+1/2})\vert_{x_2=0}}-\frac{\partial_{x_2}(u^{a\,-}+u^-_{n+1/2})\vert_{x_2=0}}{\partial_{x_2}(\Phi^{a\,-}+\Psi^-_{n+1/2})\vert_{x_2=0}}\right\}\delta\psi_n\\
&\qquad+\partial_{x_1}(\varphi^a+\psi_{n+1/2})(\delta\dot v^-_n)\vert_{x_2=0}-(\delta\dot u^-_{n})\vert_{x_2=0}=g_{n,2}-g_{n,1} .
\end{align*}
These identities suggest us to define $\delta\Psi^\pm_n$ as solutions to the following equations:
\begin{align}
\notag &\partial_t\delta\Psi^+_n+(v^{a\,+}+v^+_{n+1/2})\partial_{x_1}\delta\Psi^+_n\\
\notag &\qquad +\left\{\partial_{x_1}(\Phi^{a\,+}+\Psi^+_{n+1/2})\frac{\partial_{x_2}(v^{a\,+}+v^+_{n+1/2})}{\partial_{x_2}(\Phi^{a\,+}+\Psi^+_{n+1/2})}-\frac{\partial_{x_2}(u^{a\,+}+u^+_{n+1/2})}{\partial_{x_2}(\Phi^{a\,+}+\Psi^+_{n+1/2})}\right\}\delta\Psi^+_n\\
&\qquad +\partial_{x_1}(\Phi^{a\,+}+\Psi^+_{n+1/2})\delta\dot v^+_n-\delta\dot u^+_{n}=\mathcal R_T g_{n,2}+h^+_n , \label{delta_Psi_n+}\\[2mm]
\notag &\partial_t\delta\Psi^-_n+(v^{a\,-}+v^-_{n+1/2})\partial_{x_1}\delta\Psi^-_n\\
\notag &\qquad+\left\{\partial_{x_1}(\Phi^{a\,-}+\Psi^-_{n+1/2})\frac{\partial_{x_2}(v^{a\,-}+v^-_{n+1/2})}{\partial_{x_2}(\Phi^{a\,-}+\Psi^-_{n+1/2})}-\frac{\partial_{x_2}(u^{a\,-}+u^-_{n+1/2})}{\partial_{x_2}(\Phi^{a\,-}+\Psi^-_{n+1/2})}\right\}\delta\Psi^-_n\\
&\qquad +\partial_{x_1}(\Phi^{a\,-}+\Psi^-_{n+1/2})\delta\dot v^-_n-\delta\dot u^-_{n}=\mathcal R_T (g_{n,2}-g_{n,1})+h^-_n ,\label{delta_Psi_n-}
\end{align}
where $h^\pm_n$ are suitable source terms on $\Omega_T$, vanishing in the past and with zero traces on the boundary $\omega_T$, whose explicit form will be specified below.

In order to compute the values of $h^\pm_n$, let us make for the operator $\mathcal E$, defined in \eqref{P.new}, a decomposition similar to \eqref{decom1}. Namely we have
\begin{align}
\mathcal{E}(V_{{n}+1},\Psi_{{n}+1})-\mathcal{E}(V_{{n}},\Psi_{{n}})
&= \mathcal{E}'(V_{{n}},\Psi_{{n}})(\delta V_{{n}},\delta\Psi_{{n}})+\hat{e}_{{n}}' \notag  \\
&= \mathcal{E}'(S_{\theta_{{n}}}V_{{n}},S_{\theta_{{n}}}\Psi_{{n}})(\delta V_{{n}},\delta\Psi_{{n}})
+\hat{e}_{{n}}'+\hat{e}_{{n}}''\notag  \\
&= \mathcal{E}'(V_{{n}+1/2},\Psi_{{n}+1/2})(\delta V_{{n}},\delta\Psi_{{n}})+\hat{e}_{{n}}'
+\hat{e}_{{n}}''+\hat{e}_{{n}}'''. \label{decom3}
\end{align}
Let us denote
\begin{align} \label{e.hat}
\hat{e}_{{n}}:=\hat{e}_{{n}}'+\hat{e}_{{n}}''+\hat{e}_{{n}}''',\qquad \hat{E}_{{n}}:=\sum_{k=0}^{{n}-1}\hat{e}_{k}.
\end{align}
From \eqref{app.1}, we have
\begin{align*}
\mathcal{E}(V,\Psi)=\p_t(\Phi^a+\Psi)+(v^a+v)\partial_1(\Phi^a+\Psi)-(u^a+u)\,.
\end{align*}
Making the change of good unknown \eqref{good.NM} implies that $\mathcal{E}'(V_{{n}+1/2}^{\pm},\Psi_{{n}+1/2}^{\pm})(\delta V_{{n}}^{\pm},\delta\Psi_{{n}}^{\pm})$
are equal to the left-hand sides of  \eqref{delta_Psi_n+} and \eqref{delta_Psi_n-}, respectively.
Then it follows from \eqref{delta_Psi_n+}--\eqref{decom3}
that
\begin{align}\notag 
\mathcal{E}(V_{{n}+1},\Psi_{{n}+1})-\mathcal{E}(V_{{n}},\Psi_{{n}}) =\begin{pmatrix}
\mathcal{R}_Tg_{{n},2}+h_{{n}}^++\hat{e}_{{n}}^+\\ \mathcal{R}_T(g_{{n},2}-g_{{n},1})+h_{{n}}^-+\hat{e}_{{n}}^-
\end{pmatrix}.
\end{align}
Summing these relations and using $\mathcal{E}(V_0,\Psi_0)=0$, we get
\begin{align*}
&\mathcal{E}(V_{{n}+1}^+,\Psi_{{n}+1}^+)=\mathcal{R}_T\Big(\sum_{k=0}^{{n}}g_{k,2}\Big)+\sum_{k=0}^{{n}}h_{k}^++\hat{E}_{{n}+1}^+,\\&
\mathcal{E}(V_{{n}+1}^-,\Psi_{{n}+1}^-)=\mathcal{R}_T\Big(\sum_{k=0}^{{n}}(g_{k,2}-g_{k,1})\Big)+\sum_{k=0}^{{n}}h_{k}^-+\hat{E}_{{n}+1}^-.
\end{align*}
On the other hand, we obtain from \eqref{effective.NM} and \eqref{decom2} that
\begin{align} \label{decom2.b}
g_{{n}}=\mathcal{B}(V_{{n}+1}|_{x_2=0},\psi_{{n}+1})-\mathcal{B}(V_{{n}}|_{x_2=0},\psi_{{n}})-\tilde{e}_{{n}}.
\end{align}
In view of \eqref{P.new} and \eqref{B.def},  one obtains the relations:
\begin{align}
\big(\mathcal{B}(V_{{n}+1}|_{x_2=0},\psi_{{n}+1})\big)_2
&=\mathcal{E}(V_{{n}+1}^+|_{x_2=0},\psi_{{n}+1})\notag  \\
&=\mathcal{E}(V_{{n}+1}^-|_{x_2=0},\psi_{{n}+1})+\big(\mathcal{B}(V_{{n}+1}|_{x_2=0},\psi_{{n}+1})\big)_1. \label{B.E.relation}
\end{align}
Summing \eqref{decom2.b} and using $\mathcal{B}(V_{0}|_{x_2=0},\psi_{0})=0$, we have
\begin{align} \label{decom3.c1}
\mathcal{E}(V_{{n}+1}^-,\Psi_{{n}+1}^-)=\mathcal{R}_T\Big(\mathcal{E}\big(V_{{n}+1}^-|_{x_2=0},\psi_{{n}+1}\big)
-\widetilde{E}_{{n}+1,2}+\widetilde{E}_{{n}+1,1}\Big)+\sum_{k=0}^{{n}}h_{k}^-+\hat{E}_{{n}+1}^-.
\end{align}
Similarly, we can also obtain
\begin{align} \label{decom3.c2}
\mathcal{E}(V_{{n}+1}^+,\Psi_{{n}+1}^+)
=\mathcal{R}_T\Big(\mathcal{E}\big(V_{{n}+1}^+|_{x_2=0},\psi_{{n}+1}\big)
-\widetilde{E}_{{n}+1,2}\Big)+\sum_{k=0}^{{n}}h_{k}^++\hat{E}_{{n}+1}^+.
\end{align}

\vspace*{1mm}
\noindent{\bf Assumption\;(A-3)}: \emph{$(h_0^+,h_0^-,\hat{e}_0)=0$, and $(h_k^+,h_k^-,\hat{e}_k)$ are already given
	and vanish in the past for $k=0,\ldots,{n}-1$.}

\vspace*{2mm}
Under assumptions {{\bf (A-1)}}--{\bf (A-3)}, taking into account  \eqref{decom3.c1}--\eqref{decom3.c2}
and the property of $\mathcal{R}_T$,
we compute the source terms $h_{{n}}^{\pm}$ from
\begin{subequations} \label{source2}
	\begin{alignat}{1}
	&S_{\theta_{{n}}}\big(\hat{E}_{{n}}^+-\mathcal{R}_T\widetilde{E}_{{n},2}\big)+\sum_{k=0}^{{n}}h_{k}^+=0,\\
	&S_{\theta_{{n}}}\big(\hat{E}_{{n}}^--\mathcal{R}_T\widetilde{E}_{{n},2}+\mathcal{R}_T\widetilde{E}_{{n},1}\big)
	+\sum_{k=0}^{{n}}h_{k}^-=0.
	\end{alignat}
\end{subequations}

By virtue of assumption {\bf (A-3)} and the properties of $S_{\theta}$, it is clear that $h_{{n}}^{\pm}$ vanish in the past.
As in \cite{FM00MR1787068}, one can also check that the trace of $h_{{n}}^{\pm}$ on $\omega_T$ vanishes.
Hence, we can find $\delta\Psi_{{n}}^{\pm}$, vanishing in the past and satisfying $\delta\Psi_{{n}}^{\pm}|_{x_2=0}=\delta\psi_{{n}}$,
as the unique smooth solutions to the transport equations \eqref{delta_Psi_n+}--\eqref{delta_Psi_n-}.

Once $\delta\Psi_{{n}}$ is specified, we can obtain $\delta V_{{n}}$ from \eqref{good.NM}
and $(V_{{n}+1},\Psi_{{n}+1},\psi_{{n}+1})$ from \eqref{NM0}.
Then the errors: $e_{{n}}'$, $e_{{n}}''$, $e_{{n}}'''$, $\tilde{e}_{{n}}'$, $\tilde{e}_{{n}}''$,
$\tilde{e}_{{n}}'''$, $\hat{e}_{{n}}'$, $\hat{e}_{{n}}''$, and $\hat{e}_{{n}}'''$
are computed from \eqref{decom1}--\eqref{decom2} and \eqref{decom3},
while $e_{{n}}$, $\tilde{e}_{{n}}$, and $\hat{e}_{{n}}$ are obtained from \eqref{e.e.tilde} and \eqref{e.hat}.

Using \eqref{effective.NM} and \eqref{source},
we sum \eqref{decom1}--\eqref{decom2.b} from ${n}=0$ to $m$, respectively, to obtain
\begin{align}
\label{conv.1}&\mathcal{L}(V_{m+1},\Psi_{m+1})=\sum_{{n}=0}^{m}f_{{n}}+E_{m+1}=S_{\theta_{m}}f^a+(I-S_{\theta_{m}})E_{m}+e_m,\\
\label{conv.2}&\mathcal{B}(V_{m+1}|_{x_2=0},\psi_{m+1})=\sum_{{n}=0}^{m}g_{{n}}+\widetilde{E}_{m+1}=(I-S_{\theta_{m}})\widetilde{E}_{m}+\tilde{e}_m.
\end{align}
Plugging \eqref{source2} into \eqref{decom3.c1}--\eqref{decom3.c2}, we utilize \eqref{B.E.relation} to deduce
\begin{align}\label{conv.3}
\left\{\begin{aligned}
\mathcal{E}(V_{m+1}^-,\Psi_{m+1}^-)=&\mathcal{R}_T\big(\big(\mathcal{B}(V_{{m}+1}|_{x_2=0},\psi_{{m}+1})\big)_2
-\big(\mathcal{B}(V_{{m}+1}|_{x_2=0},\psi_{{m}+1})\big)_1\big)\\
&+(I-S_{\theta_{m}})\big( \hat{E}_m^--\mathcal{R}_T\big(\widetilde{E}_{m,2}-\widetilde{E}_{m,1}\big) \big)
+\hat{e}_{m}^--\mathcal{R}_T\big(\tilde{e}_{m,2}-\tilde{e}_{m,1}\big),\\
\mathcal{E}(V_{m+1}^+,\Psi_{m+1}^+)=&\mathcal{R}_T\big(\big(\mathcal{B}(V_{{n}+1}|_{x_2=0},\psi_{{n}+1})\big)_2\big)\\
& +(I-S_{\theta_{m}})\big( \hat{E}_m^+-\mathcal{R}_T\widetilde{E}_{m,2}\big)
+\hat{e}_{m}^+-\mathcal{R}_T\tilde{e}_{m,2}.
\end{aligned}\right.
\end{align}
Since $S_{\theta_{m}}\to I$ as $m\to \infty$, we can formally obtain the solution to problem \eqref{P.new}
from $\mathcal{L}(V_{m+1},\Psi_{m+1})\to f^a$, $\mathcal{B}(V_{m+1}|_{x_2=0},\psi_{m+1})\to 0$,
and $\mathcal{E}(V_{m+1},\Psi_{m+1})\to 0$, provided that the errors: $(e_m,\tilde{e}_m,\hat{e}_m)\to 0$.

\subsection{Inductive assumption}
Given a constant $\varepsilon>0$ and an integer $\tilde{\alpha}$ that will be chosen later on,
we assume that {{\bf (A-1)}}--{\bf (A-3)} are satisfied and that the following estimate holds:
\begin{align} \label{small}
\big\|\dot{U}^a\big\|_{H^{\tilde{\alpha}+3}_{\gamma}(\Omega_T)}+\big\|\dot{\Phi}^a\big\|_{H^{\tilde{\alpha}+4}_{\gamma}(\Omega_T)}
+\big\|{\varphi}^a\big\|_{H^{\tilde{\alpha}+7/2}_{\gamma}(\Omega_T)}+\big\|f^a\big\|_{H^{\tilde{\alpha}+2}_{\gamma}(\Omega_T)}
\leq \varepsilon.
\end{align}
Fixing another integer $\alpha$, our inductive assumption reads:
\begin{align*}
(H_{{n}-1})\ \left\{\begin{aligned}
\textrm{(a)}\,\,  &\|(\delta V_k,\delta \Psi_k)\|_{H^{s}_{\gamma}(\Omega_T)}+\|\delta\psi_k\|_{H^{s+1}_{\gamma}(\omega_T)}\leq \varepsilon \theta_k^{s-\alpha-1}\Delta_k\\
&\quad \textrm{for all } k=0,\ldots,{n}-1\textrm{ and }s\in[3,\tilde{\alpha}]\cap\mathbb{N};\\
\textrm{(b)}\,\, &\|\mathcal{L}( V_k,  \Psi_k)-f^a\|_{H^{s}_{\gamma}(\Omega_T)}\leq 2 \varepsilon \theta_k^{s-\alpha-1}\\
&\quad \textrm{for all } k=0,\ldots,{n}-1\textrm{ and } s\in[3,\tilde{\alpha}-2]\cap\mathbb{N};\\
\textrm{(c)}\,\,  &\|\mathcal{B}( V_k|_{x_2=0},  \psi_k)\|_{H^{s}_{\gamma}(\omega_T)}\leq  \varepsilon \theta_k^{s-\alpha-1}\\
&\quad \textrm{for all } k=0,\ldots,{n}-1\textrm{ and } s\in[4,\alpha]\cap\mathbb{N};\\
\textrm{(d)}\,\,  &\|\mathcal{E}( V_k,  \Psi_k)\|_{H^{3}_{\gamma}(\Omega_T)}\leq  \varepsilon \theta_k^{2-\alpha}\quad\textrm{for all } k=0,\ldots,{n}-1,
\end{aligned}\right.
\end{align*}
where $\Delta_{k}:=\theta_{k+1}-\theta_k$ with $\theta_k$ defined by \eqref{theta}.
Notice that
\begin{align*}
\frac{1}{3\theta_k}\leq \Delta_{k}=\sqrt{\theta_k^2+1}-\theta_k\leq \frac{1}{2\theta_k}\qquad\textrm{for all }k\in\mathbb{N}.
\end{align*}
In particular, sequence $(\Delta_{k})$ is decreasing and tends to $0$.
Our goal is to show that, for a suitable choice of parameters $\theta_0\geq 1$ and $\varepsilon>0$, and for $f^a$ small enough,
($H_{{n}-1}$) implies ($H_{{n}}$) and that ($H_0$) holds.
Once this goal is achieved, we infer that ($H_{{n}}$) holds for all ${n}\in \mathbb{N}$,
which enables us to conclude the proof of Theorem \ref{thm}.

From now on, we assume that ($H_{{n}-1}$) holds.
As in \cite{CS08MR2423311}, assumption ($H_{{n}-1}$) implies the following lemma.

\begin{lemma}\label{lem.triangle}
	If $\theta_0$ is large enough, then, for each $k=0,\ldots,{n}$ and each integer $s\in [3,\tilde{\alpha}]$,
	\begin{align}
	\label{tri1}&\|( V_k, \Psi_k)\|_{H^{s}_{\gamma}(\Omega_T)}+\|\psi_k\|_{H^{s+1}_{\gamma}(\omega_T)}
	\leq
	\left\{\begin{aligned}
	&\varepsilon \theta_k^{(s-\alpha)_+}\quad &&\textrm{if }s\neq \alpha,\\
	&\varepsilon \log \theta_k\quad &&\textrm{if }s= \alpha,
	\end{aligned}\right.\\
	\label{tri2}&\|( (I-S_{\theta_k})V_k, (I-S_{\theta_k})\Psi_k)\|_{H^{s}_{\gamma}(\Omega_T)}
	\leq C\varepsilon \theta_k^{s-\alpha}.
	\end{align}
	Furthermore, for each $k=0,\ldots,{n}$, and each integer $s\in [3,\tilde{\alpha}+4]$,
	one has
	\begin{align}
	\label{tri3}&\|( S_{\theta_k}V_k, S_{\theta_k}\Psi_k)\|_{H^{s}_{\gamma}(\Omega_T)}\leq
	\left\{\begin{aligned}
	&C\varepsilon \theta_k^{(s-\alpha)_+}\quad &&\textrm{if }s\neq \alpha,\\
	&C\varepsilon \log \theta_k\quad &&\textrm{if }s= \alpha.
	\end{aligned}\right.
	\end{align}
\end{lemma}

\subsection{Estimate of error terms}
To deduce ($H_{{n}}$) from ($H_{{n}-1}$), we need to estimate the quadratic errors $e'_{k}$,
$\tilde{e}_{k}'$, and $\hat{e}_{k}'$,
the first substitution errors $e_{k}''$, $\tilde{e}_{k}''$, and $\hat{e}_{k}''$,
the second substitution errors $e_{k}'''$, $\tilde{e}_{k}'''$, and $\hat{e}_{k}'''$, and the last error $D_{k+1/2} \delta\Psi_{k}$.
Recall from \eqref{decom1} that
\begin{align*}
&e_k'=\mathcal{L}(V_{k+1},\Psi_{k+1})-\mathcal{L}(V_{k},\Psi_{k})-\mathbb{L}'(U^a+V_{k},\Phi^a+\Psi_{k})(\delta V_{k},\delta\Psi_{k}),
\end{align*}
which can be rewritten as
\begin{align}
\label{error.1a}
&e_k'=\int_{0}^{1}(1-\tau)\mathbb{L}''(U^a+V_{k}+\tau \delta  V_{k},\Phi^a+\Psi_{k}+\tau \delta \Psi_{k})
\big((\delta V_{k},\delta\Psi_{k}),(\delta V_{k},\delta\Psi_{k})\big)\d\tau,
\end{align}
where operator $\mathbb{L}''$ is defined by
\begin{align*}
\mathbb{L}''(U,\Phi)\big((V',\Psi'),(V'',\Psi'')\big)
:=\left.\frac{\d}{\d \tau} \mathbb{L}'(U+\tau V'',\Phi+\tau \Psi'')(V',\Psi')\right|_{\tau=0},
\end{align*}
with operator $\mathbb{L}'$ given in \eqref{L.prime}.
We can also obtain a similar expression for $\tilde{e}_{k}'$ (resp.\;$\hat{e}_{k}'$)
defined by \eqref{decom2} (resp.\;\eqref{decom3}) in terms of the second derivative
operator $\mathbb{B}''$ (resp.\;$\mathcal{E}''$).

To control the quadratic errors, we need the following estimates for operators $\mathbb{L}''$, $\mathbb{B}''$, and $\mathcal{E}''$ (\emph{cf.}\;\eqref{error.1a}).
This can be achieved from the explicit forms of $\mathbb{L}''$, $\mathbb{B}''$, and $\mathcal{E}''$
by applying the Moser-type and Sobolev embedding inequalities.
We refer to \cite[Proposition 5]{CS08MR2423311} for the detailed proof which is omitted here for brevity.
\begin{proposition}\label{pro.tame2}
	Let $T>0$ and $s\in\mathbb{N}$ with $s\geq3$.
	Assume that $(\tilde{U},\tilde{\Phi})\in H^{s+1}_{\gamma}(\Omega_T)$
	and $\|(\tilde{U},\tilde{\Phi} )\|_{H_{\gamma}^{3}(\Omega_T)}\leq \widetilde{K}$ for all $\gamma\geq1$.
	Then there exist two positive constants $\widetilde{K}_0$ and $C$, which are independent of $T$ and $\gamma$, such that,
	if $\widetilde{K} \leq \widetilde{K}_0$, $\gamma\geq1$, and $(V_1,\Psi_1),(V_2,\Psi_2)\in H^{s+1}_{\gamma}(\Omega_T)$, then
	\begin{align}
	\notag&\big\|\mathbb{L}''\big(U^a+\tilde{U},\Phi^a+\tilde{\Phi} \big)\big((V_1,\Psi_1),(V_2,\Psi_2) \big)
\big\|_{H^{s}_{\gamma}(\Omega_T)}\\
	\notag&\quad \leq C\Big\{ \|(V_1,\Psi_1)\|_{W^{1,\infty}(\Omega_T)}\|(V_2,\Psi_2)\|_{W^{1,\infty}(\Omega_T)}
	\big\|\big(\tilde{U}+\dot{U}^a,\tilde{\Phi}+\dot{\Phi}^a\big)\big\|_{H^{s+1}_{\gamma}(\Omega_T)}\\
	\notag&\qquad\qquad  +\sum_{i\neq j}\|(V_i,\Psi_i)\|_{H^{s+1}_{\gamma}(\Omega_T)} \|(V_j,\Psi_j)\|_{W^{1,\infty}(\Omega_T)} \Big\},
	\end{align}
	and
		\begin{align}
	\notag&\big\|\mathcal{E}''\big(U^a+\tilde{U},\Phi^a+\tilde{\Phi} \big)\big((V_1,\Psi_1),(V_2,\Psi_2) \big)\big\|_{H^{s}_{\gamma}(\Omega_T)}\\\
	\notag&\quad \leq C\sum_{i\neq j}\Big\{ \|V_i\|_{L^{\infty}(\Omega_T)}\|\Psi_j\|_{H^{s+1}_{\gamma}(\Omega_T)}
	+\|V_i\|_{H^{s}_{\gamma}(\Omega_T)}	\|\Psi_j\|_{W^{1,\infty}(\Omega_T)} \Big\}.
	\end{align}
	Moreover, if $(W_1,\psi_1),(W_2,\psi_2)\in H^{s}_{\gamma}(\omega_T)\times H^{s+1}_{\gamma}(\omega_T)$, then
		\begin{align}
	\notag&\big\|\mathbb{B}''\big(U^a+\tilde{U},\Phi^a+\tilde{\Phi} \big)\big((W_1,\psi_1),(W_2,\psi_2) \big)\big\|_{H^{s}_{\gamma}(\omega_T)}\\
	\notag&\quad \leq C\sum_{i\neq j}\Big\{ \|W_i\|_{L^{\infty}(\omega_T)}\|\psi_j\|_{H^{s+1}_{\gamma}(\omega_T))}
	+\|W_i\|_{H^{s}_{\gamma}(\omega_T)}	\|\psi_j\|_{W^{1,\infty}(\omega_T)} \Big\}.
	\end{align}
\end{proposition}

Using Proposition \ref{pro.tame2}, we can obtain the following result ({\it cf}.\;\cite[Lemma 8]{CS08MR2423311}).
\begin{lemma}[Estimate of the quadratic errors] \label{lem.quad}
	Let $\alpha\geq4$. Then there exist $\varepsilon>0$ sufficiently small and $\theta_0\geq 1$ sufficiently large
	such that, for all $k=0,\ldots,{n}-1$, and all integers $s\in [3,\tilde{\alpha}-1]$,
	\begin{align}\notag
	\|(e_k', \hat{e}_k')\|_{H^{s}_{\gamma}(\Omega_T)}+\|\tilde{e}_k'\|_{H^{s}_{\gamma}(\omega_T)}
	\leq C\varepsilon^2 \theta_k^{L_1(s)-1}\Delta_k,
	\end{align}
	where  $L_1(s):=\max\{(s+1-\alpha)_++4-2\alpha,s+2-2\alpha \}$.
\end{lemma}

Now we estimate the first substitution errors $e_{k}''$, $\tilde{e}_{k}''$, and $\hat{e}_{k}''$ given in \eqref{decom1}--\eqref{decom2},
and \eqref{decom3}
by rewriting them in terms of $\mathbb{L}''$, $\mathbb{B}''$, and $\mathcal{E}''$.
For instance,  $\tilde{e}_k''$ can be rewritten as
\begin{align}\label{tilde.e''}
\notag\tilde{e}_k''=\int_{0}^{1}&\mathbb{B}''\big(U^a+S_{\theta_k}V_k+\tau(I-S_{\theta_k})V_k,\Phi^a+S_{\theta_k}\Psi_k
+\tau(I-S_{\theta_k})\Psi_k\big)\\
&\big((\delta V_k|_{x_2=0},\delta \psi_k), ((I-S_{\theta_k})V_k|_{x_2=0},(I-S_{\theta_k})\Psi_k|_{x_2=0}) \big)\d\tau.
\end{align}
Then we have the following lemma ({\it cf}.\;\cite[Lemma 9]{CS08MR2423311}).
\begin{lemma}[Estimate of the first substitution errors] \label{lem.1st}
	Let $\alpha\geq4$. Then there exist $\varepsilon>0$ sufficiently small and $\theta_0\geq 1$ sufficiently large such that,
	for all $k=0,\ldots,{n}-1$, and all integers $s\in [3,\tilde{\alpha}-2]$,
	\begin{align}\notag
	\|(e_k'', \hat{e}_k'')\|_{H^{s}_{\gamma}(\Omega_T)}
	+\|\tilde{e}_k''\|_{H^{s}_{\gamma}(\omega_T)}\leq C\varepsilon^2 \theta_k^{L_2(s)-1}\Delta_k,
	\end{align}
	where  $L_2(s):=\max\{(s+1-\alpha)_++6-2\alpha,s+5-2\alpha \}$.
\end{lemma}

Now we estimate the second substitution errors $e_{k}'''$, $\tilde{e}_{k}'''$, and $\hat{e}_{k}'''$
given in \eqref{decom1}--\eqref{decom2} and \eqref{decom3} by rewriting them in terms of $\mathbb{L}''$, $\mathbb{B}''$, and $\mathcal{E}''$.
For instance, $\hat{e}_k'''$ can be rewritten as
\begin{align}
\hat{e}_k'''=\int_{0}^{1}\mathcal{E}''\big(V_{k+1/2}+\tau(S_{\theta_k}V_k-V_{k+1/2}),\Psi_{k+1/2}\big)
\big((\delta V_k,\delta \Psi_k), (S_{\theta_k}V_k-V_{k+1/2},0) \big)\d\tau.
\end{align}
Here we have used relation $\Psi_{k+1/2}=S_{\theta_{k}}\Psi_{k}$ (\emph{cf.} \eqref{modified}).
Then one can prove the following result similar to \cite[Lemma 10]{CS08MR2423311}.

\begin{lemma}[Estimate of the second substitution errors]  \label{lem.2nd}
	Let $\alpha\geq4$. Then there exist $\varepsilon>0$ sufficiently small and $\theta_0\geq 1$ sufficiently large
	such that, for all $k=0,\ldots,{n}-1$, and all integers $s\in [3,\tilde{\alpha}-1]$, one has $\tilde{e}_k'''=0$, $ \hat{e}_k'''=0$ and
	\begin{align}\notag
	\|e_k'''\|_{H^{s}_{\gamma}(\Omega_T)}\leq C\varepsilon^2 \theta_k^{L_3(s)-1}\Delta_k,
	\end{align}
	where  $L_3(s):=\max\{(s+1-\alpha)_++8-2\alpha,s+5-2\alpha \}$.
\end{lemma}

We now estimate the last error term \eqref{error.D}:
\begin{align*}
D_{k+1/2}\delta\Psi_k=\frac{\delta\Psi_k}{\p_2(\Phi^a+\Psi_{k+1/2})}R_k,
\end{align*}
where $R_k:=\p_2\mathbb{L}(U^a+V_{k+1/2},\Phi^a+\Psi_{k+1/2})$.
This error term results from the introduction of the good unknown in decomposition \eqref{decom1}.
Note from \eqref{modified}, \eqref{small}, and \eqref{tri3} that
\begin{align*}
|\p_2(\Phi^a+\Psi_{k+1/2})|=\big|\p_2\widebar{\Phi}+\p_2\big(\dot{\Phi}^a+\Psi_{k+1/2}\big)\big|
\geq \frac{1}{2},
\end{align*}
provided that $\varepsilon$ is small enough.
Then we have the following estimate.

\begin{lemma} \label{lem.last}
	Let $\alpha\geq 5$ and $\tilde{\alpha}\geq \alpha+2$.
	Then there exist $\varepsilon>0$ sufficiently small and $\theta_0\geq 1$ sufficiently large such that,
	for all $k=0,\ldots,{n}-1$, and for all integers $s\in [3,\tilde{\alpha}-2]$, we have
	\begin{align} \label{last.e0}
	\|D_{k+1/2}\delta\Psi_k\|_{H^{s}_{\gamma}(\Omega_T)}\leq C\varepsilon^2 \theta_k^{L(s)-1}\Delta_k,
	\end{align}
	where  $L(s):=\max\{(s+2-\alpha)_++8-2\alpha,(s+1-\alpha)_++9-2\alpha,s+6-2\alpha\}$.
\end{lemma}

From Lemmas \ref{lem.quad}--\ref{lem.last},
we can immediately obtain the following estimate for $e_k$, $\tilde{e}_k$, and $\hat{e}_k$
defined in \eqref{e.e.tilde} and \eqref{e.hat}.

\begin{lemma} \label{lem.sum1}
	Let $\alpha\geq 5$. Then there exist $\varepsilon>0$ sufficiently small and $\theta_0\geq 1$ sufficiently large
	such that, for all $k=0,\ldots,{n}-1$, and for all integers $s\in [3,\tilde{\alpha}-2]$, we have
	\begin{align} \label{es.sum1}
	\|e_k\|_{H^{s}_{\gamma}(\Omega_T)}+\|\hat{e}_k\|_{H^{s}_{\gamma}(\Omega_T)}+\|\tilde{e}_k\|_{H^{s}_{\gamma}(\omega_T)}
	\leq C\varepsilon^2 \theta_k^{L(s)-1}\Delta_k,
	\end{align}
	where  $L(s)$ is defined in Lemma {\rm \ref{lem.last}}.
\end{lemma}

Lemma \ref{lem.sum1} yields the estimate of the accumulated errors $E_k$, $\widetilde{E}_k$, and $\hat{E}_k$
that are defined in \eqref{E.E.tilde} and \eqref{e.hat}.

\begin{lemma}\label{lem.sum2}
	Let $\alpha\geq 7$ and $\tilde{\alpha}=\alpha+4$.
	Then there exist $\varepsilon>0$ sufficiently small and $\theta_0\geq 1$ sufficiently large such that
	\begin{align}\label{es.sum2}
	\|(E_{{n}}, \hat{E}_{{n}})\|_{H^{\alpha+2}_{\gamma}(\Omega_T)}
	+\|\widetilde{E}_{{n}}\|_{H^{\alpha+2}_{\gamma}(\omega_T)}\leq C\varepsilon^2 \theta_{{n}}.
	\end{align}
\end{lemma}

\begin{proof}
	Observe that $L(\alpha+2)\leq 1$ if $\alpha\geq 7$. Thanks to \eqref{es.sum1}, we get
	\begin{align*}
	&\|(E_{{n}}, \hat{E}_{{n}})\|_{H^{\alpha+2}_{\gamma}(\Omega_T)}
	+\|\widetilde{E}_{{n}}\|_{H^{\alpha+2}_{\gamma}(\omega_T)}\\
	&\, \leq \sum_{k=0}^{{n}-1}\big\{\|(e_{k},\hat{e}_{k})\|_{H^{\alpha+2}_{\gamma}(\Omega_T)}
	+\|\tilde{e}_{k}\|_{H^{\alpha+2}_{\gamma}(\omega_T)} \big\}\\
	&\, \leq \sum_{k=0}^{{n}-1} C\varepsilon^2 \Delta_k
	\leq C\varepsilon^2\theta_{{n}},
	\end{align*}
	provided that $\alpha\geq 7$ and $\alpha+2\in [3,\tilde{\alpha}-2]$.
	Thus, the minimal possible $\tilde{\alpha}$ is $\alpha+4$.
\end{proof}

\subsection{Proof of Theorem \ref{thm}}

To show the main result, we first derive the estimates for source terms $f_{{n}}$, $g_{{n}}$,
and $h_{{n}}^{\pm}$ defined in \eqref{source} and \eqref{source2}; see \cite{CS08MR2423311} or \cite[Lemma 8.11]{CSW17Preprint} for the proof.

\begin{lemma} \label{lem.source}
	Let $\alpha\geq 7$ and $\tilde{\alpha}=\alpha+4$.
	Then there exist $\varepsilon>0$ sufficiently small and $\theta_0\geq 1$ sufficiently large such that,
	for all integers $s\in [3,\tilde{\alpha}+1]$,
	\begin{align}
	\label{es.fl}&\|f_{{n}}\|_{H^s_{\gamma}(\Omega_T)}
	\leq C \Delta_{{n}}\big\{\theta_{{n}}^{s-\alpha-2}(\|f^a\|_{H^{\alpha+1}_{\gamma}(\Omega_T)}
	+\varepsilon^2)+\varepsilon^2\theta_{{n}}^{L(s)-1}\big\},\\
	\label{es.gl}&\|g_{{n}}\|_{H^s_{\gamma}(\omega_T)}
	\leq C \varepsilon^2 \Delta_{{n}}\big(\theta_{{n}}^{s-\alpha-2}+\theta_{{n}}^{L(s)-1}\big),
	\end{align}
	and for all integers $s\in [3,\tilde{\alpha}]$,
	\begin{align}
	\label{es.Gl}\|h_{{n}}^{\pm}\|_{H^s_{\gamma}(\Omega_T)}\leq C\varepsilon^2 \Delta_{{n}}\big(\theta_{{n}}^{s-\alpha-2}+\theta_{{n}}^{L(s)-1}\big).
	\end{align}
\end{lemma}

Similar to \cite[Lemma 16]{CS08MR2423311} (see also \cite[Lemma 8.12]{CSW17Preprint}), we can obtain the following lemma for the solution to problem \eqref{effective.NM}
by employing tame estimate \eqref{tame}.

\begin{lemma} \label{lem.Hl1}
	Let $\alpha\geq 7$. If $\varepsilon>0$ and $\|f^a\|_{H^{\alpha+1}_{\gamma}(\Omega_T)}/\varepsilon$ are sufficiently small,
	and if $\theta_0\geq1$ is sufficiently large, then, for all integers $s\in [3,\tilde{\alpha}]$,
	\begin{align} \label{Hl.a}
	\|(\delta V_{{n}},\delta\Psi_{{n}})\|_{H^{s}_{\gamma}(\Omega_T)}+\|\delta\psi_{{n}}\|_{H^{s+1}_{\gamma}(\omega_T)}
	\leq \varepsilon \theta_{{n}}^{s-\alpha-1}\Delta_{{n}}.
	\end{align}
\end{lemma}

Estimate \eqref{Hl.a} is point (a) of $(H_{{n}})$. One can show the other inequalities in $(H_{{n}})$; see \cite[Lemmas 17-18]{CS08MR2423311} or \cite[Lemma 8.13]{CSW17Preprint} for the proof.

\begin{lemma}\label{lem.Hl2}
	Let $\alpha\geq 7$. If $\varepsilon>0$ and $\|f^a\|_{H^{\alpha+1}_{\gamma}(\Omega_T)}/\varepsilon$ are sufficiently small,
	and $\theta_0\geq1$ is sufficiently large,
	then, for all integers $s\in [3,\tilde{\alpha}-2]$,
	\begin{align}\label{Hl.b}
	\|\mathcal{L}( V_{{n}},  \Psi_{{n}})-f^a\|_{H^{s}_{\gamma}(\Omega_T)}\leq 2 \varepsilon \theta_{{n}}^{s-\alpha-1}.
	\end{align}
	Moreover, for all integers $s\in [4,\alpha]$,
	\begin{align}\label{Hl.c}
	\|\mathcal{B}( V_{{n}}|_{x_2=0},  \psi_{{n}})\|_{H^{s}_{\gamma}(\omega_T)}\leq  \varepsilon \theta_{{n}}^{s-\alpha-1}
	\end{align}
	and
	\begin{align}\label{Hl.d}
	\|\mathcal{E}( V_{{n}},  \Psi_{{n}})\|_{H^{3}_{\gamma}(\Omega_T)}\leq  \varepsilon \theta_{{n}}^{2-\alpha}.
	\end{align}
\end{lemma}

Gathering Lemmas \ref{lem.Hl1}--\ref{lem.Hl2}, we have deduced  $(H_{{n}})$ from $(H_{{n}-1})$,
provided that $\alpha\geq 7$, $\tilde{\alpha}=\alpha+4$,  \eqref{small} holds,
$\varepsilon>0$ and $\|f^a\|_{H^{\alpha+1}_{\gamma}(\Omega_T)}/\varepsilon$ are sufficiently small,
and $\theta_0\geq1$ is large enough.
Fixing constants $\alpha$, $\tilde{\alpha}$, $\varepsilon>0$ and $\theta_0\geq1$,
we can prove $(H_{0})$ as in \cite{CS08MR2423311}.
\begin{lemma}\label{lem.H0}
	If $\|f^a\|_{H^{\alpha+1}_{\gamma}(\Omega_T)}/\varepsilon$ is sufficiently small,
	then $(H_0)$ holds.
\end{lemma}

\begin{proof}[{\bf Proof of Theorem \ref{thm}}]
	We consider initial data $(U_0^{\pm},\varphi_0)$ satisfying all the assumptions of Theorem \ref{thm}.
	Let $\tilde{\alpha}=\mu-2$ and $\alpha=\tilde{\alpha}-4\geq 7$.
	Then the initial data $U_0^{\pm}$ and $\varphi_0$ are compatible up to order $\mu=\tilde{\alpha}+2$.
	From \eqref{app3} and \eqref{app5}, we obtain \eqref{small} and all the requirements
	of Lemmas \ref{lem.Hl1}--\ref{lem.H0},
	provided that $(\tilde{U}_0^{\pm},\varphi_0)$ is sufficiently small
	in $H^{\mu+1/2}(\mathbb{R}^2_+)\times H^{\mu+1}(\mathbb{R})$
	with $\tilde{U}_0^{\pm}:=U_0^{\pm}-\widebar{U}^{\pm}$.
	Hence, for small initial data, property $(H_{{n}})$ holds
	for all integers ${n}$. In particular, we have
	\begin{align*}
	\sum_{k=0}^{\infty}\left(\|(\delta V_k,\delta \Psi_k)\|_{H^{s}_{\gamma}(\Omega_T)}+\|\delta\psi_k\|_{H^{s+1}_{\gamma}(\omega_T)} \right)
	\leq C\sum_{k=0}^{\infty}\theta_k^{s-\alpha-2} <\infty\qquad \textrm{for }s\in[3,\alpha-1].
	\end{align*}
	Thus, sequence $(V_{k},\Psi_{k})$ converges
	to some limit $(V,\Psi)$ in $H^{\alpha-1}_{\gamma}(\Omega_T)$,
	and sequence $\psi_{k}$ converges to some limit $\psi$ in $H^{\alpha}_{\gamma}(\Omega_T)$.
	Passing to the limit in \eqref{Hl.b}--\eqref{Hl.c} for $s=\alpha-1=\mu-7$,
	and in \eqref{Hl.d}, we obtain \eqref{P.new}.
	Therefore, $(U, \Phi)=(U^a+V, \Phi^a+\Psi)$ is a solution on $[0,T]\times \mathbb{R}^2_+$ of the original
	problem \eqref{Phi.a}--\eqref{E0} and \eqref{Phi.b}. This completes the proof.
\end{proof}

\subsection*{Acknowledgement}
 {The research of Alessandro Morando and Paola Trebeschi was supported in part by the grant from Ministero dell'Istruzione, dell'Universit\`{a} e della Ricerca under contract PRIN 2015YCJY3A-004.}
The research of Tao Wang was supported in part by the grants from National Natural Science Foundation of China under contracts 11601398 and 11731008.

\bibliographystyle{abbrvnat}

\bibliography{nonisentropic}

\begin{thebibliography}{27}
\providecommand{\natexlab}[1]{#1}
\providecommand{\url}[1]{\texttt{#1}}
\expandafter\ifx\csname urlstyle\endcsname\relax
  \providecommand{\doi}[1]{doi: #1}\else
  \providecommand{\doi}{doi: \begingroup \urlstyle{rm}\Url}\fi

\bibitem[Alinhac(1989)]{A89MR976971}
S.~Alinhac.
\newblock Existence d'ondes de rar{\'e}faction pour des syst{\`e}mes
  quasi-lin{\'e}aires hyperboliques multidimensionnels.
\newblock \emph{Comm. Partial Diff. Eqs.}, 14\penalty0 (2):\penalty0 173--230,
  1989.
\newblock URL \url{http://dx.doi.org/10.1080/03605308908820595}.

\bibitem[Chazarain and Piriou(1982)]{CP82MR678605}
J.~Chazarain and A.~Piriou.
\newblock \emph{Introduction to the Theory of Linear Partial Differential
  Equations}.
\newblock Studies in Mathematics and its Applications, {\rm Vol. 14}.
  North-Holland Publishing Co., Amsterdam-New York, 1982.
\newblock ISBN 0-444-86452-0.
\newblock Translated from the French.

\bibitem[Chen and Wang(2008)]{CW08MR2372810}
G.-Q. Chen and Y.-G. Wang.
\newblock Existence and stability of compressible current-vortex sheets in
  three-dimensional magnetohydrodynamics.
\newblock \emph{Arch. Ration. Mech. Anal.}, 187\penalty0 (3):\penalty0
  369--408, 2008.
\newblock URL \url{http://dx.doi.org/10.1007/s00205-007-0070-8}.

\bibitem[Chen et~al.(2018)Chen, Secchi, and Wang]{CSW17Preprint}
G.-Q.~G. Chen, P.~Secchi, and T.~Wang.
\newblock Nonlinear stability of relativistic vortex sheets in
  three-dimensional {M}inkowski spacetime.
\newblock \emph{Arch. Ration. Mech. Anal.}, 2018.
\newblock ISSN 0003-9527.
\newblock URL \url{https://doi.org/10.1007/s00205-018-1330-5}.

\bibitem[Coulombel(2005)]{C05MR2138641}
J.-F. Coulombel.
\newblock Well-posedness of hyperbolic initial boundary value problems.
\newblock \emph{J. Math. Pures Appl. (9)}, 84\penalty0 (6):\penalty0 786--818,
  2005.
\newblock URL \url{http://dx.doi.org/10.1016/j.matpur.2004.10.005}.

\bibitem[Coulombel and Morando(2004)]{CM04MR2159807}
J.-F. Coulombel and A.~Morando.
\newblock Stability of contact discontinuities for the nonisentropic {E}uler
  equations.
\newblock \emph{Ann. Univ. Ferrara Sez. VII (N.S.)}, 50:\penalty0 79--90, 2004.
\newblock ISSN 0430-3202.
\newblock URL \url{https://link.springer.com/article/10.1007/BF02825344}.

\bibitem[Coulombel and Secchi(2004)]{CS04MR2095445}
J.-F. Coulombel and P.~Secchi.
\newblock The stability of compressible vortex sheets in two space dimensions.
\newblock \emph{Indiana Univ. Math. J.}, 53\penalty0 (4):\penalty0 941--1012,
  2004.
\newblock URL \url{http://dx.doi.org/10.1512/iumj.2004.53.2526}.

\bibitem[Coulombel and Secchi(2008)]{CS08MR2423311}
J.-F. Coulombel and P.~Secchi.
\newblock Nonlinear compressible vortex sheets in two space dimensions.
\newblock \emph{Ann. Sci. {\'E}c. Norm. Sup{\'e}r. (4)}, 41\penalty0
  (1):\penalty0 85--139, 2008.
\newblock URL \url{http://www.numdam.org/item?id=ASENS_2008_4_41_1_85_0}.
\newblock
  \href{http://www.irisa.fr/ipso/Semeval/DOCUMENTS/simpaf%20publi%203%20Coulombel_Secchi_ENS.pdf}{Preprint
  version}.

\bibitem[Fejer and Miles(1963)]{FM63MR0154509}
J.~A. Fejer and J.~W. Miles.
\newblock On the stability of a plane vortex sheet with respect to
  three-dimensional disturbances.
\newblock \emph{J. Fluid Mech.}, 15:\penalty0 335--336, 1963.
\newblock URL \url{https://doi.org/10.1017/S002211206300029X}.

\bibitem[Francheteau and M{{\'e}}tivier(2000)]{FM00MR1787068}
J.~Francheteau and G.~M{{\'e}}tivier.
\newblock Existence de chocs faibles pour des syst{\`e}mes quasi-lin{\'e}aires
  hyperboliques multidimensionnels.
\newblock \emph{Ast{\'e}risque}, \penalty0 (268):\penalty0 viii+198, 2000.
\newblock
  \href{https://www.math.u-bordeaux.fr/~gmetivie/chocsfaibles.pdf}{Preprint
  version}.

\bibitem[H\"ormander(1976)]{H76MR0602181}
L.~H\"ormander.
\newblock The boundary problems of physical geodesy.
\newblock \emph{Arch. Ration. Mech. Anal.}, 62\penalty0 (1):\penalty0 1--52,
  1976.
\newblock ISSN 0003-9527.
\newblock URL \url{http://dx.doi.org/10.1007/BF00251855}.

\bibitem[Lax(1957)]{L57MR0093653}
P.~D. Lax.
\newblock Hyperbolic systems of conservation laws. {II}.
\newblock \emph{Comm. Pure Appl. Math.}, 10:\penalty0 537--566, 1957.
\newblock URL \url{http://dx.doi.org/10.1002/cpa.3160100406}.

\bibitem[Lions and Magenes(1972)]{LM72MR0350178}
J.-L. Lions and E.~Magenes.
\newblock \emph{Non-homogeneous Boundary Value Problems and Applications. {\rm
  Vol. II}}.
\newblock Springer-Verlag, New York-Heidelberg, 1972.
\newblock Translated from the French by P. Kenneth, Die Grundlehren der
  Mathematischen Wissenschaften, Band 182.

\bibitem[Majda(1983{\natexlab{a}})]{M83aMR683422}
A.~Majda.
\newblock The stability of multidimensional shock fronts.
\newblock \emph{Mem. Amer. Math. Soc.}, 41\penalty0 (275):\penalty0 iv+95,
  1983{\natexlab{a}}.
\newblock ISSN 0065-9266.
\newblock \doi{10.1090/memo/0275}.
\newblock URL \url{http://dx.doi.org/10.1090/memo/0275}.

\bibitem[Majda(1983{\natexlab{b}})]{M83bMR699241}
A.~Majda.
\newblock The existence of multidimensional shock fronts.
\newblock \emph{Mem. Amer. Math. Soc.}, 43\penalty0 (281):\penalty0 v+93,
  1983{\natexlab{b}}.
\newblock ISSN 0065-9266.
\newblock \doi{10.1090/memo/0281}.
\newblock URL \url{http://dx.doi.org/10.1090/memo/0281}.

\bibitem[Majda and Osher(1975)]{MO75MR0410107}
A.~Majda and S.~Osher.
\newblock Initial-boundary value problems for hyperbolic equations with
  uniformly characteristic boundary.
\newblock \emph{Comm. Pure Appl. Math.}, 28\penalty0 (5):\penalty0 607--675,
  1975.
\newblock URL \url{http://dx.doi.org/10.1002/cpa.3160280504}.

\bibitem[M{{\'e}}tivier(2001)]{M01MR1842775}
G.~M{{\'e}}tivier.
\newblock Stability of multidimensional shocks.
\newblock In \emph{Advances in the Theory of Shock Waves}, Progr. Nonlinear
  Differential Equations Appl., {\rm Vol. 47}, pages 25--103. Birkh{\"a}user
  Boston, Boston, MA, 2001.
\newblock URL \url{http://dx.doi.org/10.1007/978-1-4612-0193-9_2}.

\bibitem[Miles(1958)]{M58MR0097930}
J.~W. Miles.
\newblock On the disturbed motion of a plane vortex sheet.
\newblock \emph{J. Fluid Mech.}, 4:\penalty0 538--552, 1958.
\newblock URL \url{https://doi.org/10.1017/S0022112058000653}.

\bibitem[Mishkov(2000)]{M00MR1781515}
R.~L. Mishkov.
\newblock Generalization of the formula of {F}aa di {B}runo for a composite
  function with a vector argument.
\newblock \emph{Int. J. Math. Math. Sci.}, 24\penalty0 (7):\penalty0 481--491,
  2000.
\newblock ISSN 0161-1712.
\newblock URL \url{http://dx.doi.org/10.1155/S0161171200002970}.

\bibitem[Morando and Trebeschi(2008)]{MT08MR2441089}
A.~Morando and P.~Trebeschi.
\newblock Two-dimensional vortex sheets for the nonisentropic {E}uler
  equations: linear stability.
\newblock \emph{J. Hyper. Diff. Eqs.}, 5\penalty0 (3):\penalty0 487--518, 2008.
\newblock URL \url{http://dx.doi.org/10.1142/S021989160800157X}.

\bibitem[Morando and Trebeschi(2013)]{MT13MR3085779}
A.~Morando and P.~Trebeschi.
\newblock Weakly well posed hyperbolic initial-boundary value problems with non
  characteristic boundary.
\newblock \emph{Methods Appl. Anal.}, 20\penalty0 (1):\penalty0 1--31, 2013.
\newblock ISSN 1073-2772.
\newblock URL \url{https://doi.org/10.4310/MAA.2013.v20.n1.a1}.

\bibitem[Morando et~al.(2018)Morando, Trakhinin, and Trebeschi]{MTT18MR3766987}
A.~Morando, Y.~Trakhinin, and P.~Trebeschi.
\newblock Local {E}xistence of {MHD} {C}ontact {D}iscontinuities.
\newblock \emph{Arch. Ration. Mech. Anal.}, 228\penalty0 (2):\penalty0
  691--742, 2018.
\newblock ISSN 0003-9527.
\newblock URL \url{https://doi.org/10.1007/s00205-017-1203-3}.

\bibitem[Rauch and Massey(1974)]{RM74MR0340832}
J.~B. Rauch and F.~J. Massey, III.
\newblock Differentiability of solutions to hyperbolic initial-boundary value
  problems.
\newblock \emph{Trans. Amer. Math. Soc.}, 189:\penalty0 303--318, 1974.
\newblock ISSN 0002-9947.
\newblock URL \url{https://doi.org/10.1090/S0002-9947-1974-0340832-0}.

\bibitem[Secchi(1996)]{S96MR1405665}
P.~Secchi.
\newblock Well-posedness of characteristic symmetric hyperbolic systems.
\newblock \emph{Arch. Ration. Mech. Anal.}, 134\penalty0 (2):\penalty0
  155--197, 1996.
\newblock ISSN 0003-9527.
\newblock URL \url{http://dx.doi.org/10.1007/BF00379552}.

\bibitem[Secchi and Trakhinin(2014)]{ST14MR3151094}
P.~Secchi and Y.~Trakhinin.
\newblock Well-posedness of the plasma-vacuum interface problem.
\newblock \emph{Nonlinearity}, 27\penalty0 (1):\penalty0 105--169, 2014.
\newblock ISSN 0951-7715.
\newblock URL \url{http://dx.doi.org/10.1088/0951-7715/27/1/105}.

\bibitem[Trakhinin(2009)]{T09MR2481071}
Y.~Trakhinin.
\newblock The existence of current-vortex sheets in ideal compressible
  magnetohydrodynamics.
\newblock \emph{Arch. Ration. Mech. Anal.}, 191\penalty0 (2):\penalty0
  245--310, 2009.
\newblock URL \url{http://dx.doi.org/10.1007/s00205-008-0124-6}.

\bibitem[Wang and Yu(2015)]{WY15MR3328144}
Y.-G. Wang and F.~Yu.
\newblock Structural stability of supersonic contact discontinuities in
  three-dimensional compressible steady flows.
\newblock \emph{SIAM J. Math. Anal.}, 47\penalty0 (2):\penalty0 1291--1329,
  2015.
\newblock URL \url{http://dx.doi.org/10.1137/140976169}.

\end{thebibliography}

\end{document}